\documentclass[a4paper]{amsart}
 
 \usepackage{amsthm}
\usepackage{amssymb}
\usepackage{latexsym}
\usepackage[mathscr]{eucal}
\usepackage{amsmath}
\usepackage{amstext}
\usepackage{amsgen}
\usepackage{amsbsy}
\usepackage{amsopn}
\usepackage{amsfonts}
\usepackage{dsfont}
\usepackage{tikz}
\usepackage{tikz-cd}
\usepackage{hyperref}

\title[Witt groups of abelian categories and perverse sheaves]
{Witt groups of abelian categories and \\perverse sheaves}

\author[J. Sch\"urmann ]{J\"org Sch\"urmann}
\address{J.  Sch\"urmann : Mathematische Institut,
          Universit\"at M\"unster,
          Einsteinstr. 62, 48149 M\"unster,
          Germany.}
\email {jschuerm@math.uni-muenster.de}

\author[J. Woolf]{Jon  Woolf}
\address{J.  Woolf : Department of Mathematical Sciences,
          University of Liverpool,
          Liverpool, L69 7ZL,
          United Kingdom.}
\email {Jonathan.Woolf@liverpool.ac.uk }

\newtheorem{theorem}{Theorem}[section]
\newtheorem{proposition}[theorem]{Proposition}
\newtheorem{corollary}[theorem]{Corollary}
\newtheorem{lemma}[theorem]{Lemma}

\theoremstyle{definition}
\newtheorem{definition}[theorem]{Definition}
\newtheorem{example}[theorem]{Example}
\newtheorem{examples}[theorem]{Examples}
\newtheorem{remark}[theorem]{Remark}
\newtheorem{remarks}[theorem]{Remarks}

\newcommand{\ie}{i.e.\ }
\newcommand{\eg}{e.g.\ }

\newcommand{\Z}{\mathbb{Z}}
\newcommand{\Q}{\mathbb{Q}}
\newcommand{\R}{\mathbb{R}}
\newcommand{\C}{\mathbb{C}}
\newcommand{\F}{\mathbb{F}}
\renewcommand{\P}{\mathbb{P}}
\renewcommand{\L}{\mathbb{L}}

\newcommand{\strat}{\mathbb{S}}

\newcommand{\im}{\mathrm{im}}

\newcommand{\cat}[1]{\mathrm{#1}}

\newcommand{\perv}[1]{\cat{Perv}(#1)}
\newcommand{\pervc}[2]{\cat{Perv}_{#1}(#2)}
\newcommand{\perva}[1]{\pervc{\text{alg}}{#1}}
\newcommand{\loc}[1]{\cat{Loc}(#1)}

\newcommand{\constr}[1]{\cat{D}^b_c(#1)}
\newcommand{\plconstr}[1]{\cat{D}^b_{pl-c}(#1)}
\newcommand{\sh}[1]{\mathcal{#1}}

\newcommand{\id}{\text{id}}

\newcommand{\mor}[2]{{\mathrm{Hom}}(#1,#2)}

\newcommand{\ext}[3]{{\mathrm{Ext}^{#1}}(#2,#3)}

\newcommand{\ih}[2]{I\!H^{#1}\!\left(#2\right)}
\newcommand{\ic}[1]{\mathcal{I}\mathcal{C}({#1})}

\newcommand{\shcoh}[2]{H^{#1}(#2)}

\newcommand{\pf}[1]{{}^p\!{#1}}

\newcommand{\pt}{{\text{pt}}}

\newcommand{\ME}{\Xi^\text{un}_f}
\newcommand{\VC}{\Phi^\text{un}_f}
\newcommand{\NC}{\Psi^\text{un}_f}

\newcommand{\wg}[2]{\widetilde{\rm W}_{#1}(#2)}
\newcommand{\wgc}[2]{{\rm W}_{#1}(#2)}

\newcommand{\mw}[2]{\widetilde{\rm MW}_{#1}(#2)}
\newcommand{\mwc}[2]{{\rm MW}_{#1}(#2)}

\newcommand{\oc}[2]{#1^{#2}}
\newcommand{\ir}[2]{#1 \lhd #2}

\newcommand{\defn}[1]{\emph{#1}}

\newcommand{\codim}{\mathrm{codim}\, }

\newcommand{\mono}{\hookrightarrow}
\newcommand{\epi}{\twoheadrightarrow}
\newcommand{\coker}{\mathrm{coker}\, }
\renewcommand{\im}{\mathrm{im}\, }
\newcommand{\coim}{\mathrm{coim}\, }
\newcommand{\colim}{\mathrm{colim}\, }

\newcommand{\rep}[1]{ \mathrm{Rep}\left( #1 \right) }

\begin{document}

\begin{abstract} 
In this paper we study the Witt groups $W_{\pm}(\perv{X})$ of perverse sheaves on a finite-dimensional topologically stratified space $X$ with even dimensional strata.
 We show that $W_{\pm}(\perv{X})$ has a canonical decomposition as a direct sum of the Witt groups of shifted local systems on strata. We compare this with another `splitting decomposition' for  Witt classes of perverse sheaves obtained inductively from our main new tool, a `splitting relation' which is a generalisation of isotropic reduction. 
 
 The Witt groups $W_{\pm}(\perv{X})$ are identified with the (non-trivial) Balmer--Witt groups of the constructible derived category $\constr{X}$ of sheaves on $X$, and also with the corresponding cobordism groups defined by Youssin.

Our methods are primarily algebraic and apply more widely. The general context in which we work is that of a triangulated category with duality, equipped with a self-dual $t$--structure with noetherian heart, glued from self-dual $t$--structures on a thick subcategory and its quotient. 
\end{abstract}

\maketitle

\section{Introduction}

 The signature of a compact, oriented manifold is a basic topological invariant. It is an obstruction to the existence of a null-bordism, and plays a key role in surgery theory and the classification of manifolds. The signature can be extended to singular spaces by using intersection cohomology --- a compact Witt space $W$ is a space whose rational intersection cohomology satisfies Poincar\'e duality and $\sigma(W)$ is defined to be the signature of the associated intersection form. { For example any irreducible complex analytic or algebraic variety  is a Witt space.} A more refined invariant is the Witt class $w(W)$ of the intersection form in the rational Witt group $\wgc{}{\Q}$. This determines the signature but also contains torsion information which is localised on the singularities of the space. For manifolds, and more generally for spaces with integral Poincar\'e duality such as integral homology manifolds and intersection Poincar\'e spaces \cite{MR1057240}, this torsion information vanishes and the Witt class is simply the signature. The Witt class is the obstruction to the existence of a Witt null-bordism \cite{siegel}. It plays an analogous r\^ole in stratified surgery theory and the classification of stratified spaces to that played by the signature for manifolds.

In this paper we study the Witt group $W(\perv{X})$ of perverse sheaves. Here $X$ is a finite-dimensional topologically stratified space with even dimensional strata, 
and $\perv{X}$ the category of perverse sheaves, constructible with respect to the stratification, with rational coefficients. A proper stratified map  $f\colon W \to X$ from a Witt space $W$ { with $\dim W \equiv 0 \:(4)$} determines a class {  $[f_*I_W]\in W(\perv{X})$} whose pushforward to $W(\perv{\pt}) \cong W(\Q)$ { for $X$ compact} is $w(W)$.
{ Here 
\[
I_W\colon \ic{W} \to D \ic{W}
\]
 is the symmetric intersection form of the corresponding intersection cohomology complex of $W$, with $D$ the Verdier duality for constructible sheaf complexes.}
 Thus $W(\perv{X})$ is the natural home for relative invariants of spaces over $X$.

The category $\perv{X}$ is constructed by `gluing together' categories of shifted local systems on the strata of $X$. As a consequence $W(\perv{X})$ decomposes as a direct sum of the Witt groups of shifted local systems --- see Corollary \ref{witt decomposition theorem}. We refer to the associated decomposition of a class as the \defn{canonical decomposition}. 
In \S\ref{stratified spaces} we give an algorithm, { starting from a top-dimensional open stratum}, see (\ref{direct sum formula}), for computing the canonical decomposition of a class. The algorithm relies on the ability to identify maximal isotropic subobjects of forms on local systems, so its feasibility depends on the complexity of the fundamental groups of the strata of $X$.  { We are also interested in the structure of the Witt group $W(\perv{X})$ itself, which by the above can be reduced to the simpler and more classical case of Witt groups of  local systems, for instance see Example~\ref{ex:Bunke-Ma} for the case of real coefficients and all strata orientable. If for example all strata $S$ of $X$ are simply connected, therefore orientable, then Corollary \ref{witt decomposition theorem} implies
\begin{equation}\label{simply-connected}
W(\perv{X}) \cong \bigoplus_{S \colon \dim S \equiv 0\:(4)} W(\Q).
\end{equation}
In particular $W(\perv{X}) = 0$ when all strata $S$ of $X$ are simply connected with $\dim S \equiv 2\:(4)$. In the above mapping situation this implies $[f_*I_W]=0$ and hence $w(W)=0\in W(\Q)$. See \cite[Theorem~6.1]{MR1127478} for the corresponding vanishing of the signature $\sigma(W)$.\\}

Cappell and Shaneson state an expression for a Witt class as a sum of classes of forms on intersection cohomology complexes  \cite[Theorem 2.1]{cs}, see \cite[Chapter 8]{Banagl}. To be a little more precise, they obtain a decomposition for a class in { their cobordism group $\Omega_{CS}(X)$  of symmetric} self-dual complexes, but we show the latter is isomorphic to the Witt group of perverse sheaves --- see Proposition \ref{abelian-triangulated relation} and Corollary \ref{witt-cobordism relation} { : 
\[
W(\perv{X})  \cong  \Omega_{CS}(X) \:.
\]
}
They view this decomposition as an up-to-cobordism topological analogue of the { following famous Decomposition Theorem:
\begin{theorem}\label{alg-decomposition-theorem} Let $f: W\to X$ be a proper stratified morphism of complex algebraic varieties, with $W$ irreducible and all strata $S$ of $X$ also complex algebraic. Then
\begin{enumerate}
\item (Decomposition:) $Rf_*\ic{W}\cong \oplus_i \: {}^pR^if_*\ic{W}$ is isomorphic to the direct sum of the corresponding perverse direct image sheaves.
\item (Strict support:) Each perverse direct image sheaf for $i\in \Z$ is a direct sum  ${}^pR^if_*\ic{W}\cong  \oplus_S$  $\ic{\overline{S};\mathcal{L}_{i,S}}$ of twisted intersection cohomology sheaf complexes on the closures $\overline{S}$ of the strata $S$.
\item (Semisimplicity:) The local system $\mathcal{L}_{i,S}$ on $S$ is semisimple for all $i$ and $S$.
\end{enumerate}
\end{theorem} 

\begin{remark} This Decomposition Theorem is due to Beilinson, Bernstein, Deligne  and Gabber \cite[Th\'eor\`eme 6.2.5]{bbd} via arithmetic techniques and results for perverse sheaves on schemes over a finite base field. Another proof and far reaching extension, even applying for a projective morphism of complex analytic varieties, was given by M. Saito \cite{MR1000123, MHM} via his theory of pure and mixed Hodge modules. Finally, in the complex algebraic context a more geometric proof was found by de Cataldo and Migliorini \cite{dCM}. We refer to the beautiful survey \cite{dCM2} for more details, as well as to \cite[Introduction]{BW17} for a short overview of the recent extension of the Decomposition Theorem to semisimple perverse sheaf complexes.
\end{remark} }

In { our topological context we obtain in analogy to (2) and (3) above} a decomposition up to isomorphism for anisotropic forms on perverse sheaves, but only up to Witt equivalence in general. 
{ In fact the perverse sheaves underlying an anisotropic form are semisimple (see Corollory~\ref{direct-sum-cor} for the corresponding algebraic result in a noetherian abelian category with duality).}
 The perverse sheaves underlying pure algebraic Hodge modules automatically carry anisotropic forms coming from polarisations \cite[\S 5.2]{MR1000123}. 
{ Similarly, polarizations of Hodge structures for suitable topological intersection pairings appear inductively in the proof of \cite{dCM}.} This explains why one has a stronger result when working in the algebraic as opposed to in our topological context. 

In our notation the \defn{Cappell--Shaneson decomposition} is (\ref{main eqn}) below. Since intersection cohomology complexes are precisely the intermediate extensions of local systems on the strata it makes sense to compare the canonical and Cappell--Shaneson decompositions. Before doing so though, we should mention that  there is an error in their proof, and (\ref{main eqn}) needs correcting for stratifications of depth greater than or equal to two. The depth one results cited in \cite[Theorem 4.2]{MR2646988} and \cite{MR2796414} are correct. An explicit counterexample for a depth two stratification is provided in \S\ref{rankstrat} using a quiver description for perverse sheaves on rank stratifications  \cite{braden1999}. 

Using a different method of proof we obtain a new, more complicated, expression (\ref{cs1}) which reduces to Cappell and Shaneson's in certain cases, { e.g. for anisotropic forms on perverse sheaves
(see Proposition~\ref{anisotropic formula}).} The key ingredient in the proof is { the following `splitting relation' for Witt classes: 
Let $\imath \colon  Y \hookrightarrow X$ be the inclusion of a closed stratified subspace, in other words $Y$ is a closed union of strata of $X$, with $\jmath \colon  U=X-Y \hookrightarrow X$  the complementary open inclusion.  
Suppose  $\beta\colon \mathcal{B}\to D\mathcal{B}$ is a non-degenerate symmetric form in $\perv{X}$. Then 
\begin{equation}\label{split-perv}
[\beta] = [\imath_*\imath^{!*}\beta] + [\jmath_{!*}\jmath^*\beta].
\end{equation}
in the Witt group  $W(\perv{X})$. Here $$\jmath_{!*} = \im (\pf{\jmath_!} \to \pf{\jmath_*})\quad \text{ and} \quad \imath^{!*} = \im (\pf{\imath^!} \to \pf{\imath^*})$$ are respectively the intermediate extension and restriction. 
In the depth one case in which $Y$ and $U$ are topological manifolds and the strata are their connected components, this reduces to the Cappell--Shaneson decomposition (\ref{main eqn}) below. In this case the perverse truncation used to define the intermediate  restriction $\imath^{!*} = \im (\pf{\imath^!} \to \pf{\imath^*})$ is just truncation of a  sheaf complex with respect to the standard $t$-structure. In general however, the intermediate restriction uses the more complicated perverse truncation, which cannot be expressed easily in geometric terms. By  iterated application of the `splitting relation' (\ref{split-perv}) we end up with  our new decomposition (\ref{cs1}), which we refer to as the \defn{splitting decomposition}. This involves iterated intermediate restrictions.  It turns out that the splitting decomposition (\ref{cs1})} is {\em not} the canonical decomposition in general. Moreover, it can depend upon the choice of representative for the Witt class and on a choice of ordering of the strata of $X$. The reason for these negative results is that intermediate extension is not an exact functor. When it is one obtains stronger results, in particular { (see Corollary~\ref{cs formula conditions}):}

\begin{corollary}
If each stratum has finite fundamental group, or if certain (twisted) intersection cohomology groups of links vanish, then the splitting decomposition (\ref{cs1}) is the canonical one. Moreover, under the second vanishing condition it simplifies to Cappell and Shaneson's decomposition
\begin{equation}
\label{main eqn}
[\beta] = \sum_\textrm{strata $S$} [{\imath_S}_*{\jmath_S}_{!*}\jmath_S^*\imath_S^{!*}\beta]
\end{equation}
where $\imath_S \colon  \overline{S} \hookrightarrow X$ and $\jmath_S \colon   S \hookrightarrow \overline{S}$ are the inclusions.
\end{corollary}

{ \begin{remark}
In the complex algebraic context the results of this paper don't contribute any new information to the Decomposition Theorem, except that the decomposition of ${}^pR^0f_*\ic{W}$  in Theorem~\ref{alg-decomposition-theorem} fits with the Cappell--Shaneson decomposition as well as with our canonical and splitting decompositions, because the induced form $f_*I_W$ on ${}^pR^0f_*\ic{W}$ is anisotropic. 

In our topological context the canonical decomposition comes from the direct sum decomposition of  $W(\perv{X})$, and is very helpful for understanding the structure of this Witt group.
However, the canonical decomposition of $f_*I_W$ in the stratified mapping situation $f \colon W\to X$ for a Witt space $W$ is very difficult to understand in terms of the geometry of $f$, since one has to find an anisotropic representative in the Witt class $[f_*I_W]$.

Cappell and Shaneson \cite{cs} give a nice geometric interpretation of their decomposition (\ref{main eqn}), but this may differ from our canonical decomposition and only holds under additional assumptions. Our splitting decomposition can be viewed as a technical tool to relate the Cappell--Shaneson and canonical decompositions, when the former holds.   
\end{remark}

\begin{remark}
Cappell and Shaneson introduce the notion of a `locally nonsingular' self-dual perverse sheaf and show in \cite[Theorem~3.2]{cs} that such a `locally nonsingular' self-dual perverse  sheaf is isometric to an orthogonal direct sum of forms on twisted intersection cohomology complexes $\ic{\overline{S};\mathcal{L}_{S}}$ as in part (2) of the Decomposition Theorem~\ref{alg-decomposition-theorem}. This result can also be shown by induction (starting from a closed stratum of smallest dimension) via the `splitting criterion' of \cite[Lemma~4.1.3 and  Remark~4.1.2]{cs} as in the approach of de Cataldo and Migliorini to the Decomposition Theorem. This corresponds to a decomposition of a perverse sheaf as a direct sum of twisted intersection cohomology complexes $\ic{\overline{S};\mathcal{L}_{S}}$, similar to the `strict support decomposition' of pure Hodge modules in 
\cite[(5.1.3.5) and Lemma~5.1.4]{MR1000123}. It implies the Cappell--Shaneson decomposition (\ref{main eqn}) in the Witt group, but it need not correspond to the canonical decomposition, because here one doesn't require the local systems $\mathcal{L}_{S}$ to be semisimple; cf.\ Example~\ref{abstract-nonsingular} and equation (\ref{explicit sum}) for abstract algebraic counterparts. In particular the notion of a `locally nonsingular' self-dual perverse sheaf is weaker than that of an  anisotropic form on a perverse sheaf.
\end{remark} }

For the purposes of this introduction we have framed the above results in a geometric context. However, our methods are primarily algebraic and apply more widely, { see Examples \ref{triangulated cats with duality} and \ref{triangulated gluing examples}}. The general context in which we work is that of a triangulated category with duality, and a self-dual $t$--structure glued from self-dual $t$--structures on a thick subcategory and its quotient. Our first main result, Proposition \ref{abelian-triangulated relation}, identifies $W(\perv{X})$ with the zero'th Balmer--Witt group of the constructible derived category $\constr{X}$ of sheaves on $X$:
{ 
\begin{equation}\label{Witt-Witt}
W(\perv{X}) \cong W_0(\constr{X})\:.
\end{equation} 
This implies many functorial properties of the Witt group $W(\perv{X}) $ of perverse sheaves.  For a stratified map $f\colon W\to X$ from a compact Witt space $W$ it implies
\[
[f_*I_W]=f_*[I_W] \in W(\perv{X}) \cong W_0(\constr{X})
\]
is the direct image of the symmetric intersection form $[I_W]\in W_0(\constr{W})$ under the pushforward $f_*=f_!$, which commutes with Verdier duality. In a sense this is the substitute for part (1) of the Decomposition Theorem~\ref{alg-decomposition-theorem} in our topological context.}

 When $X$ is compact and admits a triangulation compatible with the stratification, for instance when $X$ is a compact Whitney or subanalytic stratified space, then we can pass to the zero'th Balmer--Witt group of the PL-constructible derived category. With $\Q$ coefficients, these Witt groups form a generalised homology theory isomorphic to symmetric $L$-theory \cite[Corollary 4.10]{witty}. Our splitting decomposition therefore induces formul\ae\ for the $L$-theoretic fundamental classes { $ [\beta]_{\L}$ of self-dual perverse sheaves
\[
W(\perv{X}) \cong W_0(\constr{X}) \to W_0(\plconstr{X}) \colon [\beta]\mapsto [\beta]_{\L}
\]
 as sums of forms on simple perverse sheaves.  In our approach it is important to start with the constructible derived category with respect to a fixed stratification, with its self-dual perverse $t$-structure, since the latter is not visible in the PL context. Such formul\ae \ for $L$-theoretic fundamental classes of self-dual perverse sheaves were foreseen by Cappell--Shaneson~\cite{cs} as natural improvements of their formul\ae \ for homological $L$-classes of self-dual perverse sheaves. This simple definition of the  $L$-theoretic fundamental classes of self-dual perverse sheaves needs the identification~(\ref{Witt-Witt}) with Balmer-Witt groups, and not just the cobordism groups $\Omega_{CS}(X)$ of Cappell--Shaneson~\cite{cs} (or~\cite{youssin}).\\}

Pushing forward to a point one obtains corresponding formul\ae\ for signatures { and Witt classes} of self-dual perverse sheaves. These generalise the classical Chern--Hirzebruch--Serre formula for the signature of a smooth fibre bundle to singular spaces and perverse sheaves on them.
{ Since this is not the subject of this paper, we only  illustrate it by the following simple example of a compact oriented base manifold $X$ as a one stratum space. Let $f\colon W \to X$  be proper stratified map  from a Witt space $W$ with $\dim W \equiv 0 \:(4)$ to a  compact oriented manifold $X$ of even dimension. The fibre $F$ of $f$ is also a Witt space. Assume ${}^pR^0f_*\ic{W}$ is a constant local system on $X$, i.e. $\pi_1X$ acts trivially on the middle dimensional intersection cohomology $\ih{(\dim W - \dim X)/2}{F}$. Then 
\[
w(W)=\sigma(X)\cdot w(F)\in W(\Q)\:.
\]
}

The main tool we use is the aforementioned `splitting relation' (Theorem \ref{sum thm}) which is a generalisation of isotropic reduction. This is expressed most naturally in terms of degenerate forms, and so in \S\ref{witt groups} we review the construction of the Witt group of an abelian category explaining how to treat degenerate forms on an equal footing with non-degenerate ones. The Witt class of a degenerate form is the class of the induced non-degenerate form on its image; for this reason it is essential that we work with abelian categories rather than in the broader context of exact categories where there is no notion of image.

Our main results are consequences  of the splitting relation. Firstly, it implies
\[
\wgc{}{\cat{A}} \cong \bigoplus_{[s\cong Ds]} \wgc{}{\langle s \rangle}
\]
where  $\cat{A}$ is a noetherian abelian category, $\langle s \rangle$ is the full Serre subcategory generated by the self-dual simple object $s\cong Ds$, and the sum is over isomorphism classes of such objects. { This is well-known, see for example \cite[\S 6]{MR543249} or \cite[Chapter 5]{sheiham}, although the usual proof uses Hermitian devissage rather than our splitting relation.
See also \cite[Corollary 4.13]{youssin}}, but note that the $\wgc{}{\langle s\rangle}$ need not be freely generated as claimed there. Secondly, when
\[
\cat{A} \stackrel{\imath_*}{\to} \cat{B} \stackrel{\jmath^*}{\to} \cat{C}
\]
 is an  exact triple of triangulated categories with duality and the self-dual $t$-structure on $\cat{B}$ is glued from $t$-structures on $\cat{A}$ and $\cat{C}$, the splitting relation yields a formula
 \[
 [\beta] = [\imath_*\imath^{!*}\beta] + [ \jmath_{!*}\jmath^*\beta]
 \]
 in $W(\cat{B}^0)$ where $\cat{B}^0$ is the self-dual heart of the $t$-structure. In general this formula depends upon the representative form $\beta$.

In \S\ref{stratified spaces} we apply these algebraic results to categories of perverse sheaves on a topologically stratified space with finitely many strata. The splitting decomposition (\ref{cs1}) is obtained by iteratively applying the splitting relation: we choose an ordering of the strata and split off terms on an open stratum one-by-one. In \S\ref{rankstrat} and \S\ref{schubertstrat} we provide some explicit examples and counterexamples using the quiver description of perverse sheaves on a rank stratification given in \cite{braden1999} and on Schubert-stratified projective spaces given in \cite{MR1900761}.

In the final section we turn our attention to algebraically constructible perverse sheaves $\pervc{\text{alg}}{X}$ on a complex algebraic variety $X$. If $f\colon X \to \C$ is an algebraic map then the unipotent nearby and vanishing cycles formalism of \cite{Beilinson1987} provides an equivalence between this and a `gluing category' built from $\pervc{\text{alg}}{f^{-1}(0)}$ and $\pervc{\text{alg}}{X-f^{-1}(0)}$. In this situation too the Witt group decomposes as a direct sum
\[
W(\pervc{\text{alg}}{X}) \cong W( \pervc{\text{alg}}{X-f^{-1}(0)} ) \oplus W( \pervc{\text{alg}}{f^{-1}(0)} ).
\]
The projection is given by restriction along $\jmath\colon X-f^{-1}(0) \hookrightarrow X$ and the perverse unipotent vanishing cycles functor $\VC$, and the inclusions are given by the maximal extension functor $\ME$ and extension by zero along $\imath \colon f^{-1}(0) \hookrightarrow X$. Corollary \ref{relating gluing and splitting} relates this decomposition to the terms in the splitting formula, specifically
\begin{align*}
[\imath^{!*}\beta] &= \VC [\beta] - [\NC(\jmath^*\beta)\circ N] \\
 [\jmath_{!*}\gamma] &= \ME [\gamma] + \imath_*[\NC\gamma \circ N]
\end{align*}
where $\NC$ is the perverse unipotent nearby cycles functor, and $N:\NC\to \NC(-1)$ is, up to a Tate twist, the logarithm  of the  monodromy $\mu$ acting on  $\NC$.\\

\noindent
{ {\bf Acknowledgments.} 
We dedicate this paper to the memory of Andrew Ranicki.

We would like to thank M. Banagl, S. Cappell, G. Friedman, L. Maxim and J. Shaneson for helpful discussions. We also thank the referees for their comments and suggestions.

J. Sch\"{u}rmann was funded by the Deutsche Forschungsgemeinschaft (DFG, German Research Foundation) under Germany's Excellence Strategy --- EXC 2044, Mathematics M\"{u}nster: Dynamics-Geometry-Structure.}

\section{Witt groups}
\label{witt groups}

\subsection{Categories with duality}
\label{categories with duality}
A \defn{category with duality} is a triple $(\cat{A}, D,\chi)$ in which $\cat{A}$ is a category, $D$ is a functor
$\cat{A}^{\rm op} \to \cat{A}$, and $\chi$ is a natural isomorphism $\id \to D^2$ such that the morphisms 
\[
\begin{tikzcd}
Da\ar{rr}{\chi_{Da}} && D^3a & \textrm{and} & D^3a \ar{rr}{D\chi_{a}}&& Da
\end{tikzcd}
\]
are mutually inverse for any object $a \in \cat{A}$.
\begin{examples}
We are principally interested in abelian categories with duality. These arise in many contexts in topology, geometry and representation theory, usually related to finite-dimensional representations of some (graded) algebra with involution. Prominent examples include
\begin{enumerate}
\item local systems on  { a topological manifold $M$ (in the connected case these are modules over the group ring of the fundamental group 
$\pi_1M$,} with involution induced by the group inverse);
\item finite-dimensional representations of a quiver with involution (as in \cite[\S3.2]{Young});
\item finitely-generated torsion modules over a Dedekind ring $R$.  
\end{enumerate}
In each case the duality is given by morphisms into a dualising object; in the first { case this is the orientation sheaf $or_M$ of $M$ --- if $M$ is connected and oriented this is the trivial representation of the fundamental group; in the second case it is the constant one-dimensional representation;} in the third case it is $Q(R)/R$ where $Q(R)$ is the quotient field.
\end{examples}
A \defn{bilinear form} on an object $a\in\cat{A}$ is a morphism $\alpha \colon  a \to Da$. A form is \defn{non-degenerate} if $\alpha$ is an isomorphism, and it is \defn{$\epsilon$-symmetric}, where $\epsilon$ is either $+1$ or $-1$, if the diagram
\[
\begin{tikzcd}
a \ar{rr}{\alpha} \ar{dr}[swap]{\chi(a)} && Da \\
& D^2 a \ar{ur}[swap]{\epsilon D\alpha} 
\end{tikzcd}
\]
commutes. To make sense of $\epsilon$-symmetry we need $\cat{A}$ to be additive. In fact it suffices to consider the case $\epsilon=1$ since we may always absorb the sign into the definition of the natural transformation $\chi$, \ie antisymmetric forms are symmetric forms for a different duality. 

Fix a bilinear form $\beta\colon b \to Db$.  Given a morphism $f\colon a\to b$ the \defn{restriction} $\beta|_f$ is the composite  $Df\circ \beta \circ f$ on $a$. When $f$ is a monomorphism we will often abuse notation and denote the restriction by $\beta|_a$. The restriction $\beta|_f$ is symmetric whenever $\beta$ is. 

Bilinear forms $\alpha$ and $\beta$ are \defn{isometric}, written $\alpha\cong \beta$, if there is an isomorphism $f\colon  a \to b$ such that $\alpha=\beta|_f$. For example, when $\alpha\colon a\to Da$ is non-degenerate then $(D\alpha)^{-1}=\epsilon \chi(a) \alpha^{-1}$ is a  symmetric form and is isometric to $\alpha$ because 
\[
\begin{tikzcd}
a \ar{d}[swap]{\alpha} \ar{r}{\alpha} & Da \ar{d}{(D\alpha)^{-1}} \\
Da & D^2a \ar{l}{D\alpha}
\end{tikzcd}
\]
commutes. Isometry is an equivalence relation which preserves non-degeneracy and symmetry. The \defn{Witt monoid  of degenerate forms} $\mw{}{\cat{A}}$ is the set of isometry classes of symmetric forms under direct sum. The non-degenerate symmetric forms constitute a sub-monoid, the \defn{Witt monoid} $\mwc{}{\cat{A}}$. 

Suppose $(\cat{A}, D_\cat{A},\chi_\cat{A})$ and $(\cat{B}, D_\cat{B},\chi_\cat{B})$ are categories with duality, and $F\colon \cat{A} \to \cat{B}$ a functor. We say $F$ \defn{commutes with duality} if there is a natural isomorphism $\eta\colon  FD_\cat{A} \to D_\cat{B}F$ such that 
\[
\begin{tikzcd}
F \ar{r}{F\chi} \ar{d}[swap]{\chi F}& F D_\cat{A}^2\ar{d}{\eta D_\cat{A}}\\
D_\cat{B}^2 F\ar{r}[swap]{D_\cat{B} \eta} & D_\cat{B}FD_\cat{A}
\end{tikzcd}
\]
commutes. This ensures that  $\eta_aF\alpha$ is symmetric for $D_\cat{B}$ whenever $\alpha\colon a\to Da$ is symmetric for $D_\cat{A}$. Such a functor induces a morphism
$\mw{}{\cat{A}}\to \mw{}{\cat{B}}$ 
which restricts to a morphism between the submonoids of non-degenerate forms. We will suppress the natural transformation $\eta_a$ and simply write $F\alpha$ for the image form. 

\subsection{Witt groups of abelian categories}
\label{witt groups of abelian categories}

Suppose that $\cat{A}$ is an abelian category with exact duality $D$. It follows that if $\ker f \mono a$ is a kernel of $f\colon a\to b$ then $Da \epi D\ker f$ is a cokernel of $Df\colon Db \to Da$. Therefore there is a canonical isomorphism $D\ker f \cong \coker Df$, and similarly $D\,\coker f \cong \ker Df$. In practice we will suppress these identifications.

Fix a symmetric form $\beta \colon  b \to Db$. A subobject $\imath\colon a \mono b$ is 
\begin{enumerate}
\item \defn{$\beta$-isotropic} if the restriction $\beta|_\imath=0$;
\item  \defn{$\beta$-lagrangian} if the sequence $0 \to a \stackrel{\imath}{\longrightarrow} b \stackrel{D\imath  \beta}{\longrightarrow} Da \to 0$ is exact;
\item and \defn{$\beta$-null} if $\beta\circ \imath=0$.
\end{enumerate}
When the form $\beta$ is understood we suppress it from the notation. Null and lagrangian subobjects are isotropic, but not necessarily {\it vice versa}. Isotropic subobjects are also known, for instance in \cite{handbook}, as \defn{sublagrangians} because any subobject of a lagrangian is isotropic. If a form has no non-zero isotropic subobjects we say it is \defn{anisotropic}. 

The \defn{orthogonal complement} of a subobject $\imath \colon  a \mono b$ is defined to be the subobject
\[
\oc{a}{\beta} = \ker (D\imath \beta).
\]
A subobject $\imath\colon a \mono b$ is isotropic if and only if it factors through $\oc{a}{\beta}$, lagrangian if and only if the factorisation is an isomorphism $a \cong \oc{a}{\beta}$ and null if and only if the inclusion is an isomorphism $\oc{a}{\beta} \cong b$. 

A non-degenerate form $\eta$ which has a lagrangian is called {\em metabolic}. Non-degenerate forms $\beta_0$ and $\beta_1$ are {\em Witt-equivalent} if they are stably isometric by metabolic forms, \ie if there exist metabolic forms $\eta_0$ and $\eta_1$ such that
\[
\beta_0 \oplus \eta_0 \cong \beta_1 \oplus \eta_1.
\]
This defines an equivalence relation on $\mwc{}{\cat{A}}$.
\begin{definition}
The  \defn{Witt group}   $\wgc{}{\cat{A}}$ of $\cat{A}$ is the set of Witt-equivalence classes in $\mwc{}{\cat{A}}$ under $\oplus$. This is a group, not just a monoid, because $\beta\oplus -\beta$ is Witt equivalent to $0$. The class of a non-degenerate symmetric form $\beta$ is denoted $[\beta]$.
\end{definition}
\begin{remark}
Making the analogous definitions with antisymmetric forms in place of symmetric ones or, as explained above, working with symmetric forms in the category with duality $(\cat{A}, D, -\chi)$, we obtain the Witt  group $\wgc{-}{\cat{A}}$ of antisymmetric forms.
\end{remark}

If $F\colon \cat{A}\to \cat{B}$ is an exact functor which commutes with duality then it preserves metabolic forms and so induces maps $\wgc{\pm}{F}\colon \wgc{\pm}{\cat{A}} \to {  \wgc{\pm}{\cat{B}} }$. We will see shortly that  in some cases we can weaken the requirement that $F$ is exact.

\subsection{Isotropic reduction}
\label{isotropic reduction}

Fix a symmetric form $\beta \colon  b \to Db$. Given a null subobject $\imath\colon a \mono b$ there is an induced symmetric form on the cokernel of $\imath$ such that
\[
\begin{tikzcd}
b \ar{d}{\beta} \ar[two heads]{r} & \coker \imath \ar[dashed]{d}\\
Db & \ker D\imath \ar[hook]{l}
\end{tikzcd}
\]
commutes --- symmetry follows from the uniqueness of the induced morphism. In particular $\ker \beta$ is always null and the corresponding symmetric form $\overline{\beta} \colon  \im \beta \to \coim D\beta$ is non-degenerate. 

This is a special case of a more general construction starting from an isotropic subobject $\imath\colon a \mono b$. Note that the factorisation $a \mono \oc{a}{\beta}$ is always null  for the restriction $\beta|_{\oc{a}{\beta}}$ because 
\[
D(\oc{a}{\beta}) =D \ker(D\imath\beta) \cong \coker (D\beta D^2\imath) \cong \coker (\beta \imath).
\]
{ It is a kernel of $\beta|_{\oc{a}{\beta}}$ when $\beta$ is non-degenerate.} The \defn{isotropic reduction} $\ir{\beta}{a}$ is defined to be the induced symmetric form on the cokernel of $a \mono \oc{a}{\beta}$.  We note some special cases:  when $\beta$ is non-degenerate $\ir{\beta}{a} = \overline{\beta|_{\oc{a}{\beta}}}$, when $a$ is a null subobject $\ir{\beta}{a}$ is the induced symmetric form on the quotient, and in particular $\ir{\beta}{\ker \beta}=\overline{\beta}$.  The isotropic reduction is the zero form on the zero object if, and only if, $\imath\colon a \to b$ is lagrangian. If $\beta$ is  non-degenerate then so is any reduction of $\beta$ (but not {\it vice versa}).

Isotropic reduction is compatible with restriction to a subobject in the following sense.
\begin{lemma}
\label{restriction and reduction commute}
Suppose we have a commutative diagram 
\[
\begin{tikzcd}
a \ar[hook]{r}{\imath} \ar{d}{0} & b \ar[hook]{r}{\jmath} \ar{d}{\beta} & c \ar{d}{\gamma} \\
Da & Db \ar[two heads]{l}{D\imath}& Dc \ar[two heads]{l}{D\jmath}
\end{tikzcd}
\]
in which $\gamma \colon  c \to Dc$ is symmetric (so that $\beta=\gamma|_\jmath$ and $a$ is an isotropic subobject of both $\beta$ and of $\gamma$). Then there is a { monomorphism $\ir{\jmath}{a} \colon  \oc{a}{\beta}/a \to \oc{a}{\gamma}/a$} such that
\[
\ir{\left(\gamma|_\jmath\right)}{a}=\left(\ir{\gamma}{a}\right)|_{\ir{\jmath}{a}}.
\]
\end{lemma}
\begin{proof}
Taking successive pullbacks we obtain a commutative diagram:
\[
\begin{tikzcd}
a \ar[hook]{r} \ar[equal]{d} & \oc{a}{\beta} \ar[hook]{r} \ar[hook]{d} & \oc{a}{\gamma} \ar[hook]{d} \\
a  \ar[hook]{r}{\imath} & b  \ar[hook]{r}{\jmath} & c.
\end{tikzcd}
\]
Hence there is an induced { monomorphism $\ir{\jmath}{a}\colon  \oc{a}{\beta}/a \to \oc{a}{\gamma}/a$} such that the top (and dual bottom) inner squares of the following diagram commute.
\[
\begin{tikzcd}
\oc{a}{\beta}/a \ar[hook]{rrr}{\ir{\jmath}{a}} \ar{ddd}[swap]{\ir{\beta}{a}}  &&& \oc{a}{\gamma}/a \ar{ddd}{\ir{\gamma}{a}} \\
& \oc{a}{\beta} \ar[hook]{r} \ar{d}[swap]{\beta|_{\oc{a}{\beta}}} \ar[two heads]{ul} & \oc{a}{\gamma} \ar{d}{\gamma|_{\oc{a}{\gamma}}} \ar[two heads]{ur}&\\
& D\oc{a}{\beta} & D\oc{a}{\gamma} \ar[two heads]{l}&\\
D\left(\oc{a}{\beta}/a\right)\ar[hook]{ur} &&& D\left(\oc{a}{\gamma}/a\right) \ar[two heads]{lll}{D(\ir{\jmath}{a})} \ar[hook]{ul}
\end{tikzcd}
\]
The remaining internal squares commute by definition. Hence the outer square commutes. 
\end{proof}
Reduction by the kernel of a degenerate form is compatible with isotropic reductions in the following sense. 
\begin{lemma}
\label{commuting reductions}
Suppose $\imath\colon a \mono b$ is isotropic for symmetric $\beta\colon  b \to Db$. Then 
\[
\overline{\ir{\beta}{a}} \cong \ir{\overline{\beta}}{\overline{a}}
\]
where $\overline{a}$ is the image of $a \mono b \epi \im \beta$.
\end{lemma}
\begin{proof}
Let $\overline{b}=\im\beta$ and $\overline{a}=\im(a \mono b \epi \im \beta)$. Then there is a commutative diagram:
\[
\begin{tikzcd}
a \ar{ddd}{0} \ar[hook]{rr} \ar[two heads]{dr} && b \ar[two heads]{d} & \oc{a}{\beta} \ar[two heads]{d} \ar[hook]{l}  \ar[two heads]{rr} && \oc{a}{\beta}/a \ar[two heads]{dl} \ar{ddd}{\ir{\beta}{a}}\\
&\overline{a} \ar[hook]{r} \ar{d} & \overline{b} \ar{d}{\overline{\beta}} & \oc{\overline{a}}{\overline{\beta}}  \ar{d} \ar[hook]{l} \ar[two heads]{r}& \oc{\overline{a}}{\overline{\beta}}/\overline{a} \ar{d}{\ir{\overline{\beta}}{\overline{a}}}&\\
&D\overline{a}  \ar[hook]{dl} & D\overline{b} \ar[hook]{d} \ar[two heads]{r}  \ar[two heads]{l} & D\oc{\overline{a}}{\overline{\beta}}  \ar[hook]{d}& D\left(\oc{\overline{a}}{\overline{\beta}}/\overline{a}\right) \ar[hook]{l}\ar[hook]{dr}&\\
Da && Db \ar[two heads]{r} \ar[two heads]{ll} & D \oc{a}{\beta} && D\left(\oc{a}{\beta}/a\right) \ar[hook]{ll}
\end{tikzcd}
\]
The result follows by considering the right hand square and recalling that $\ir{\overline{\beta}}{\overline{a}}$ is an isomorphism.
\end{proof}
Loosely we can say that `reduction by the kernel commutes with all other isotropic reductions'. 

The proof of the next lemma is an elementary diagram chase, which we omit. 
\begin{lemma}
\label{reduction in stages}
Let $\gamma \colon  c \to Dc$ be symmetric and $a \mono c$ isotropic. Then quotienting by $a$ induces a one-to-one correspondence between factorisations $a \mono b  \mono c$ with $b$ isotropic and isotropic subobjects of the reduced form $\ir{\gamma}{a}$. Furthermore $\ir{\gamma}{b} \cong \ir{\left(\ir{\gamma}{a}\right)}{(b/a)}$.
\end{lemma}

 Since the reduction of a non-degenerate form is non-degenerate, isotropic reduction generates an equivalence relation on $\mwc{}{\cat{A}}$. 
 \begin{theorem}[{See \eg \cite[Theorem { 1.1.32} and Remark { 1.1.33}]{handbook}}]
The equivalence relation on $\mwc{}{\cat{A}}$ generated by isotropic reduction is Witt-equivalence. Hence the set of equivalence classes is $\wgc{}{\cat{A}}$.
\end{theorem}
Although isotropic reduction is often only considered for non-degenerate forms it is a natural operation on degenerate forms too. Let $\wg{}{\cat{A}}$ be the set of equivalence classes of the relation generated by isotropic reduction on $\mw{}{\cat{A}}$. Reduction by the kernel defines a map of monoids
\[
\mw{}{\cat{A}} \to \mwc{}{\cat{A}} \colon  \beta \mapsto\overline{\beta}.
\]
By Lemma \ref{commuting reductions} this map preserves the equivalence relation generated by isotropic reduction.  Hence there are maps 
\[
\wgc{}{\cat{A}} \to \wg{}{\cat{A}} \to \wgc{}{\cat{A}}
\]
induced by $\mwc{}{\cat{A}} \hookrightarrow \mw{}{\cat{A}}$ and reduction by the kernel respectively.
\begin{corollary}
\label{deg or nondeg}
These maps are inverse to one another. Hence $\wg{}{\cat{A}}$ is also a group under $\oplus$ and it is isomorphic to the Witt group $\wgc{}{\cat{A}}$. 
\end{corollary}
\begin{proof}
In one direction the composition is the identity on representatives, and in the other it is isotropic reduction by the kernel. Both induce the identity on Witt groups. 
\end{proof}
Thus one can define the Witt group by using the isotropic reduction relation on either degenerate or on non-degenerate forms.

\subsection{The splitting relation}
\label{the splitting relation}

In this section we introduce a more general relation which allows us to split forms into two pieces. Isotropic reduction corresponds to the special case when one of these pieces is trivial.

\begin{proposition}
\label{splitting relation}
Suppose $\beta \colon  b \to Db$ is a non-degenerate symmetric form and that $0\to a \stackrel{\imath}{\longrightarrow} b \stackrel{q}{\longrightarrow} c \to 0$ is a short exact sequence in $\cat{A}$. Then there are induced symmetric forms $\alpha=\beta|_\imath \colon  a \to Da$ and $\gamma = \beta|_{\beta^{-1}Dq}\colon  Dc \to D^2c$ and 
\[
[\beta]  = [\overline{\alpha}] + [\overline{\gamma}] 
\]
in the Witt group $\wgc{}{\cat{A}}$.
\end{proposition}
\begin{proof}
There is a unique isomorphism $f\colon \ker \alpha \to \ker\gamma$ such that
\[
\begin{tikzcd}
\ker \alpha \ar[dashed]{rrrr}{f} \ar[hook]{d} &&&& \ker \gamma  \ar[hook]{d}\\
a \ar{d}{\alpha} \ar[hook]{rr}{\imath}  && b \ar{d}{\beta} && Dc \ar{d}{\gamma} \ar[hook, swap]{ll}{\beta^{-1}Dq}\\
Da && Db \ar[two heads]{ll}{D\imath} \ar[two heads, swap]{rr}{\chi(c)q\beta^{-1}}&& D^2c  
\end{tikzcd}
\]
commutes. Let $k=\ker\alpha\cong \ker \gamma$. We can apply Lemma \ref{restriction and reduction commute} simultaneously to both lower squares of the above diagram to obtain a new diagram
\[
\begin{tikzcd}
\im\alpha \ar{d}{\overline{\alpha}} \ar[hook]{rr}  && \oc{k}{\beta}/k \ar{d}{\ir{\beta}{k}} && \im\gamma \ar{d}{\overline{\gamma}} \ar[hook]{ll}\\
\coim D\alpha && D\left( \oc{k}{\beta}/k\right) \ar[two heads]{ll} \ar[two heads]{rr}&&\coim D\gamma  
\end{tikzcd}
\]
in which the vertical arrows are isomorphisms. Furthermore we can check from the construction of Lemma \ref{restriction and reduction commute} that the diagonal 
\[
\im \alpha \mono \oc{k}{\beta}/k \epi \coim D\gamma
\]
 of this new diagram is still short exact (and the other diagonal is the dual short exact sequence). 

Thus we can reduce to the special case in which $\alpha$ and  $\gamma$ are non-degenerate. In this case $\left( \imath \quad \beta^{-1}Dq\right) \colon  a \oplus Dc \to b$ is an isomorphism and
\[
\left(\begin{array}{c} D\imath \\ D^2q D\beta^{-1}\end{array}\right) \beta \left(\begin{array}{ll} \imath & \beta^{-1}Dq\end{array}\right)
= \left(\begin{array}{ll}  \alpha & 0 \\ 0 & \gamma \end{array}\right).
\]
So $\beta\cong \alpha\oplus \gamma$ and $[\beta]  = [\alpha] + [\gamma]$. More generally this argument shows that $[\beta] = [\ir{\beta}{k}] = [\overline{\alpha}]+[\overline{\gamma}]$.
\end{proof}
\begin{remarks}
\label{splitting remarks}
\begin{enumerate}
\item In the situation of the above lemma $\oc{a}{\beta}\cong Dc$ and the restricted form $\gamma$  is isometric to $\beta|_{\oc{a}{\beta}}$. Hence the splitting relation can be written in $\wg{}{\cat{A}}$ as
\[
[\beta] = [\beta|_a] + [\beta|_{\oc{a}{\beta}}].
\]
\item The proposition shows that the splitting relation holds in $\wgc{}{\cat{A}}$. Conversely we could define $\wgc{}{\cat{A}}$ using the splitting relation, for both the relation of isometry and that arising from isotropic reduction are special cases obtained respectively by putting $c=0$ and $\alpha=0$.
\item If $\beta$ is anisotropic then the proof provides an isometry $\beta \cong \alpha \oplus \gamma$, in particular $\alpha$ and $\gamma$ are also non-degenerate and anisotropic.
\end{enumerate}
\end{remarks}

The following result is a minor generalisation of the splitting relation. 
\begin{corollary}
\label{splitting cor}
Suppose $\beta \colon  b \to Db$ is a non-degenerate symmetric form and that $a \stackrel{f}{\longrightarrow} b \stackrel{g}{\longrightarrow} c$ is exact at the middle term. Then there are induced symmetric forms $\alpha=\beta|_f \colon  a \to Da$ and $\gamma = \beta|_{\beta^{-1}Dg}\colon  Dc \to D^2c$ such that $[\beta]  = [\overline{\alpha}] + [\overline{\gamma}]$ in the Witt group $\wgc{}{\cat{A}}$.
\end{corollary}
\begin{proof}
Replacing $\alpha$ by $\ir{\alpha}{\ker f}$ and $\gamma$ by $\ir{\gamma}{\ker Dg}$ we are in the situation of Proposition \ref{splitting relation}. Hence, using Lemma \ref{commuting reductions},
\begin{align*}
[\beta] &= \left[\,\overline{\ir{\alpha}{\ker f}}\,\right] + \left[\,\overline{\ir{\gamma}{\ker Dg}}\,\right]\\
&= \left[\, \ir{\overline{\alpha}}{\overline{\ker f}} \right] + \left[\,\ir{\overline{\gamma}}{\overline{\ker Dg}}\,\right]\\
&= \left[\, \overline{\alpha}\right] + \left[\, \overline{\gamma}\,\right].
\end{align*}
\end{proof}

In the presence of an exact duality the following are equivalent (the last two by the Jordan--H\"older theorem):  
\begin{enumerate}
\item $\cat{A}$ is noetherian;
\item $\cat{A}$ is artinian;
\item $\cat{A}$ is artinian and noetherian;
\item $\cat{A}$ is a length category, \ie each object has a finite composition series with simple factors.
\end{enumerate}
Under these conditions the Witt group has a more explicit description. 
\begin{corollary}
\label{simples generate witt gp}
Suppose $\cat{A}$ is noetherian. Then the Witt group $\wgc{}{\cat{A}}$ is the set of isometry classes of anisotropic forms. The group operation is given by choosing an anisotropic representative for the direct sum. Any anisotropic form is isometric to a direct sum of non-degenerate symmetric forms on simple objects of $\cat{A}$. In particular the Witt group is generated by forms on simple objects.
\end{corollary}
\begin{proof}
If $\beta \colon  b \to Db$ is a symmetric form then Lemma \ref{reduction in stages} and the noetherian property guarantee that there is a maximal isotropic subobject $a \mono b$. The reduction $\ir{\beta}{a}$ is thus an anisotropic representative for $[\beta]$. Youssin \cite[Theorem 4.9]{youssin} shows that anisotropic forms represent the same Witt class if and only if they are isometric. (In other words, even though the Witt Cancellation Theorem may not hold, its conclusion remains true for anisotropic forms.) Finally, by the third part of Remarks \ref{splitting remarks}, and another application of the noetherian property, we can write an anisotropic form as a finite direct sum of forms on simple objects.  
\end{proof}
The Witt group is not necessarily freely generated by forms on simple objects (as claimed in \cite{youssin}) as can be seen by considering, for example,  the categories of vector spaces over $\Q$ or $\C$ whose Witt groups have torsion. However, it does have a canonical direct sum decomposition into Witt groups of the Serre subcategories generated by self-dual simple objects. This is well-known, see for example \cite[\S 6]{MR543249} or \cite[Chapter 5]{sheiham}, although the usual proof uses Hermitian devissage rather than our splitting relation.
\begin{corollary}
\label{direct-sum-cor}
Suppose $\cat{A}$ is noetherian. Then there is an isomorphism 
\[
\wgc{}{\cat{A}} \cong \bigoplus_{[s\cong Ds]} \wgc{}{\langle s \rangle}
\]
where the direct sum is over isomorphism classes of self-dual simple objects and  $\langle s \rangle$ denotes the full Serre subcategory generated  by self-extensions of $s$.
\end{corollary}
\begin{proof}
Suppose $s$ is a self-dual simple object. Then the duality $D$ restricts to a duality on the full Serre subcategory $\langle s \rangle$ and the inclusion $\imath_s$ is an exact functor which commutes with duality. Hence there are induced maps $\wgc{}{\imath_s}\colon  \wgc{}{\langle s \rangle} \to \wgc{}{\cat{A}}$ and combining these a map
\[
\bigoplus_{[s\cong Ds]} \wgc{}{\langle s \rangle} \to \wgc{}{\cat{A}}.
\]
It is surjective by the last part of Corollary \ref{simples generate witt gp}. Moreover, the description of the Witt group as isometry classes of anisotropic forms shows that it is injective; an isometry must preserve the summand consisting of forms on self-extensions of a given simple object.
\end{proof}

\subsection{Balmer--Witt groups of triangulated categories}
\label{triangulated witt groups}
A triangulated category $\cat{B}$ with duality has $4$-periodic Balmer--Witt groups. Proposition \ref{abelian-triangulated relation} below expresses the Witt groups of the abelian heart of a self-dual $t$-structure on $\cat{B}$ in terms of the Balmer--Witt groups of $\cat{B}$. This is closely related to \cite[Theorem 4.3]{twg2} which treats the special case in which the triangulated category is the bounded derived category of the heart (but which works in the more general setting of the derived category of an exact category). See also \cite[Theorem 7.4]{youssin} where the analogous result is proved for a slightly different definition of triangulated Witt group. 

Suppose $\cat{B}$ is triangulated with shift functor $[1]$. Exact triangles in $\cat{B}$ will be denoted either by $a\to b\to c \to a[1]$ or by a diagram
\[
\begin{tikzcd}
a \ar{rr} && b \ar{dl} \\
& c \ar[dashed]{ul} & 
\end{tikzcd}
\]
where the dotted arrow denotes a map $c \to a[1]$. In order that the Balmer--Witt groups of $\cat{B}$ are defined and well-behaved we will always assume that
\begin{enumerate}
\item  $\cat{B}$ is essentially small so that isomorphism classes of objects form a set;
\item  $\cat{B}$ satisfies the \emph{enriched octahedral axiom};
\item $2$ is invertible in $\cat{B}$, \ie given $\alpha \in \mor{a}{b}$ there exists $\alpha'$ with $\alpha = 2 \alpha'$.
\end{enumerate}
As noted in \cite[Remarque 1.1.13]{bbd} and \cite{twg1} the second property is satisfied by all commonly met triangulated categories, in particular by derived categories. It also passes to triangulated subcategories and to localisations. 

Suppose that $D$ is a triangulated duality on $\cat{B}$ with natural transformation $\chi\colon  \id \to D^2$. Then one can define Balmer--Witt groups $W_i(\cat{B})$ for $i\in \Z$, see \cite{twg1} but note that we use homological indexing rather than cohomological so that our $W_i(\cat{B})$ corresponds to Balmer's $W^{-i}(\cat{B})$. The group $W_0(\cat{B})$ is the quotient of the Witt monoid by the submonoid generated by  \defn{metabolic} forms (or \defn{neutral} forms in the terminology of \cite{twg1}), \ie non-degenerate forms $\beta\colon b\to Db$ for which there is a \defn{lagrangian} $\alpha\colon  a\to b$ such that the triangle
\[
\begin{tikzcd}
a \ar{r}{\alpha} & b \ar{rr}{D\alpha\cdot \beta}&& Da \ar{r}{\gamma}&  a[1]
\end{tikzcd}
\]
is exact  and $\gamma$ is symmetric, \ie $(D\gamma)[1] = \chi_{a[1]} \gamma$. The group $W_i(\cat{B})$ is defined similarly but using the shifted duality $c\mapsto (Dc)[-i]$ with natural isomorphism 
\[
(-1)^{i(i-1)/2}\chi\colon \id \to (D[-i])^2.
\]
 Although the shifted duality is not triangulated when $i$ is odd it is still a $\delta$-functor, and this suffices for the construction. In contrast to the abelian case $[\beta]=0$ if and only if $\beta$ is metabolic \cite[Theorem 3.5]{twg1}. There are natural isomorphisms $W_i(\cat{B})\cong W_{i+4}(\cat{B})$ given by $[\beta] \mapsto [\beta[-2]]$ so that the groups are $4$-periodic. The Balmer--Witt groups are functorial under triangulated functors which commute with duality since these preserve metabolic forms.
 
Recall that a \defn{$t$-structure} on $\cat{B}$ is a strict, full subcategory $\cat{B}^{\leq 0}\subset \cat{B}$ such that $\cat{B}^{\leq 0}[1] \subset \cat{B}^{\leq 0}$ and for each $c\in \cat{B}$ there is an exact triangle
\[
\tau^{\leq 0}c \to c \to \tau^{>0}c \to \tau^{\leq 0}c [1]
\]
with $\tau^{\leq 0}c\in \cat{B}^{\leq 0}$ and $\tau^{>0}c\in \cat{B}^{> 0}$ where the latter is the full subcategory on those objects $c$ such that $\mor{b}{c}=0$ for all $b\in \cat{B}^{\leq 0}$. Indeed the existence of these triangles implies that $\cat{B}^{\leq 0}$ is right admissible, with right adjoint $\tau^{\leq 0}$ to its inclusion and that $\cat{B}^{> 0}$ is left admissible with left adjoint $\tau^{>0}$ to its inclusion. These adjoints are referred to as truncation functors. The exact triangle associated to an object $c$ is unique (up to isomorphism) and the first two maps in it come respectively from the counit and unit of the adjunctions. 

Let $\cat{B}^{\leq n} = \cat{B}^{\leq 0}[-n]$ with left adjoint $\tau^{\leq n}$ to its inclusion, and define $\cat{B}^{\geq n}$ and the right adjoint $\tau^{\geq n}$ to its inclusion similarly. The subcategory $\cat{B}^0=\cat{B}^{\leq 0} \cap \cat{B}^{\geq 0}$ is abelian \cite[Th\'eor\`eme 1.3.6]{bbd} and is known as the \defn{heart} of the $t$-structure. The functor $H^0 = \tau^{\leq0}\tau^{\geq 0}\colon \cat{B} \to \cat{B}^0$ is cohomological, \ie takes exact triangles to long exact sequences. 

A triangulated functor $F\colon  \cat{B} \to \cat{C}$ between categories with respective $t$-structures $\cat{B}^{\leq 0}$ and $\cat{C}^{\leq 0}$ is \defn{left $t$-exact} if $F\cat{B}^{\geq 0} \subset \cat{C}^{\geq 0}$, \defn{right $t$-exact} if $F\cat{B}^{\leq 0} \subset \cat{C}^{\leq 0}$ and \defn{$t$-exact} if it is both left and right $t$-exact. The induced functor ${}^pF :=H^0F$ between the abelian hearts --- the peculiar notation arises from the original occurrence in \cite{bbd} of these notions in the context of perverse sheaves --- is respectively left exact, right exact and exact accordingly. 

If $D$ is an exact duality on $\cat{B}$ then one can check that $D(\cat{B}^{\geq 0})$ is also a $t$-structure. We refer to this as the dual $t$-structure, and say a $t$-structure is \defn{self-dual} if $\cat{B}^{\leq 0}=D(\cat{B}^{\geq 0})$. The duality $D$ restricts to an exact duality on the heart of a self-dual $t$-structure. Conversely, if the $t$-structure is bounded and the heart is invariant under duality then the $t$-structure is self-dual. If in addition the heart is a length category then this is equivalent to the set of simple objects being invariant under duality. 

\begin{proposition}
\label{abelian-triangulated relation}
Suppose $\cat{B}$ is a triangulated category with exact duality $D$ and $\cat{B}^0$ is the heart of a self-dual $t$-structure on $\cat{B}$. Then 
\[
W_i(\cat{B}) \cong 
\begin{cases}
W(\cat{B}^0) & i=0 \mod 4\\
W_-(\cat{B}^0) & i=2 \mod 4\\
0 & i=1 \mod 2.
\end{cases}
\]
\end{proposition}
\begin{proof}
We treat the case $i=0$ first. The inclusion $\cat{B}^0 \hookrightarrow \cat{B}$ commutes with duality and preserves metabolic forms. Therefore it induces a map $W(\cat{B}^0)\to W_0(\cat{B})$. The functor $H^0\colon  \cat{B}\to \cat{B}^0$ also commutes with duality. Whenever $\alpha\colon  a \to b$ is a lagrangian for $\beta$ there is an exact sequence
\[
\cdots \to H^0a \to H^0b \to H^0Da \to \cdots.
\]
Hence we can apply Corollary \ref{splitting cor} to deduce that $[H^0\beta]=0\in W(\cat{B}^0)$. Therefore there is an induced map $W_0(\cat{B}) \to W(\cat{B}^0)\colon  [\beta]\mapsto[H^0\beta]$. It is clear that the composite 
\[
W(\cat{B}^0)\to W_0(\cat{B})\to W(\cat{B}^0)
\]
is the identity (in fact on representatives). We also claim that $[\beta]=[H^0\beta]$ in $W(\cat{B})$, from which it follows immediately that $W(\cat{B}^0)\cong W_0(\cat{B})$. To establish the claim we use the sub-lagrangian construction of \cite[\S4]{twg1}, which is an analogue of isotropic reduction for the triangulated setting. 

Given a non-degenerate symmetric form $\beta\colon  b \to Db$ in $\cat{B}$ we observe that $\imath\colon \tau^{<0}b \to b$ is isotropic (or in the terminology of \cite{twg1} sub-lagrangian) because $D\imath\beta\imath=0$. Furthermore, the natural morphism $\jmath\colon  \tau^{<0}b \to \tau^{\leq 0}b$  is a `good morphism' in the sense of \cite[Definition 4.3]{twg1} because there exist morphisms $q$ and $r$ such that the diagram (in which we omit some natural morphisms)
\[
\begin{tikzcd}
\tau^{<0}b \ar{d}{\jmath} \ar{r}{\imath} & b \ar{d}{\beta} \ar{r}{q}& \tau^{\geq0}Db \ar{d}{D\jmath} \ar{r}{r[1]} & \tau^{<0}b[1] \ar{d}{\jmath[1]}\\
\tau^{\leq 0}b \ar{r}{Dq}& Db \ar{r}{D\imath}& \tau^{>0}Db \ar{r}{Dr}& \tau^{\leq 0}b[1]
\end{tikzcd}
\]
is commutative. Indeed by applying the enhanced octrahedral axiom to the octahedron below (which  for ease of reading we draw as upper and lower halves with dotted arrows indicating the boundary morphisms of the exact triangles and labels on morphisms omitted)
\[
\begin{tikzcd}
H^0Db \ar{rr} \ar{dr} && \tau^{<0} b[1]  \ar{dd} &
H^0Db \ar{rr} && \tau^{<0} b[1] \ar{dd}\ar{dl} \\
& \tau^{\geq 0}Db \ar{ur}\ar{dl} &
&&\tau^{\leq0}b[1] \ar[dashed]{ul}\ar{dr}& \\
 \tau^{>0}Db \ar[dashed]{uu}&& Db[1] \ar[dashed]{ll} \ar[dashed]{ul}  &
  \tau^{>0}Db  \ar{ur} \ar[dashed]{uu} && Db[1] \ar[dashed]{ll} 
  \end{tikzcd}
\]
we obtain the required triangle which shows that $\jmath$ is a `very good morphism' in the sense of \cite[Definition 4.11]{twg1}. Applying  \cite[Theorem 4.20]{twg1} we deduce that $[\beta]=[H^0\beta]$ as required.

The other cases follow more easily: the group $W_1(\cat{B})$ vanishes because each representative $b \to (Db)[-1]$ has a lagrangian, namely $\tau^{\leq 0}b \to b$; the case $i=2$ is similar to $i=0$, but with symmetric forms replaced by antisymmetric ones, and the case $i=3$ is similar to  $i=1$. We omit the details.
\end{proof}
 \begin{corollary}
 \label{witt-cobordism relation}
 Suppose $\cat{B}$ is a triangulated category with exact duality $D$ and $\cat{B}^0$ is the heart of a self-dual $t$-structure on $\cat{B}$. Then there are canonical isomorphisms
 \[
 W_0(\cat{B}) \cong \Omega_+(\cat{B}) \quad \textrm{and} \quad W_2(\cat{B}) \cong \Omega_-(\cat{B})
 \]
 between the non-vanishing Balmer--Witt groups and the Youssin cobordism groups $\Omega_\pm(\cat{B})$ of symmetric and antisymmetric self-dual complexes introduced in \cite{youssin}.
 \end{corollary}
 \begin{proof}
 There are canonical surjective homomorphisms 
 \[
 W_0(\cat{B}) \to \Omega_+(\cat{B}) \qquad \text{and} \qquad W_2(\cat{B}) \to \Omega_-(\cat{B}),
 \]
  see  \cite[p31]{MR2646988}. These are compatible with the isomorphisms to $W_\pm(\cat{B}^0)$ provided by Proposition \ref{abelian-triangulated relation} and \cite[Theorem 7.4]{youssin}.
 \end{proof}

\begin{examples}
\label{triangulated cats with duality}
The results above apply in various interesting examples:
\begin{enumerate}
\item The bounded derived category of an abelian category with duality, with its evident induced duality and standard $t$-structure.
\item The constructible derived category of sheaves of vector spaces on a finite-dimensional topologically stratified space, with only even-dimensional strata, equipped with Verdier duality and the self-dual perverse $t$-structure, see for example \cite[\S 2]{bbd} and \cite[\S4.2]{MR2031639}. Here, by topologically stratified space we mean a locally cone-like stratified space in the sense of Siebenmann, see for example \cite[\S 4.2]{MR2031639}.
\item The constructible derived category of sheaves of torsion modules over a Dedekind ring $R$ on a finite-dimensional topologically stratified space, with only even-dimensional strata, as studied in \cite{MR1127478}. The torsion condition is preserved by push-forward along open inclusions because the stalks of the push-forward can be expressed in terms of the compact link \cite[Remark 4.4.2]{MR2031639}, and by the K\"unneth formula they vanish after tensoring with $Q(R)$. It follows from \cite[\S3.3]{bbd} that the perverse $t$-structure is self-dual for shifted Verdier duality.
\item 
\label{flops}
Let $f \colon X \to Y$ be a proper morphism of complex algebraic varieties, of fibre dimension at most $1$ and with $Rf_* \mathcal{O}_X=\mathcal{O}_Y$. Then the standard $t$-structure restricts to a $t$-structure on the null category $C_f$, i.e.\ the full category of $D^b \textrm{Coh}(X)$ on objects $\sh{E}$ with $Rf_*\sh{E}=0$, see \cite[Lemma 3.1]{MR1893007}. If in addition $f$ is an isomorphism outside a subvariety of dimension $0$ then the heart $C_f \cap \textrm{Coh}(X)$ is stable under shifted Grothendieck duality by \cite[Proposition 9.7 and Theorem 9.8]{bodzenta-flops}. 
\end{enumerate}

\end{examples}

\subsection{Gluing and splitting}
\label{gluing and splitting}

Suppose that $\cat{A} \stackrel{\imath_*}{\longrightarrow} \cat{B} \stackrel{\jmath^*}{\longrightarrow} \cat{C}$ is an exact triple of triangulated categories, \ie $\imath_*$ is the inclusion of a full, thick triangulated subcategory $\cat{A}$ of a triangulated category $\cat{B}$, and $\cat{C}$ is the quotient category obtained by localising at all morphisms in $\cat{A}$. If $\cat{B}$ has a triangulated duality which preserves the subcategory $\cat{A}$ then both $\cat{A}$ and $\cat{C}$ inherit triangulated dualities such that the inclusion $\imath_*$ and quotient $\jmath^*$ commute with duality. Theorem 6.2 of \cite{twg1} states that there is then a long exact sequence
\[
\cdots \to W_i(\cat{A}) \to W_i(\cat{B}) \to W_i(\cat{C}) \to W_{i-1}(\cat{A}) \to \cdots
\]
of Balmer--Witt groups in which the first two maps are induced from $\imath_*$ and $\jmath^*$ respectively (the hypothesis in \cite{twg1} that $\cat{C}$ is weakly cancellative is unnecessary, see \cite[Theorem 2.1]{MR1886007}).

For the remainder of this section we suppose further that 
\begin{enumerate}
\item $\imath_*$ has respective left and right adjoints $\imath^*$ and $\imath^!$;
\item $\jmath^*$ has respective left and right adjoints $\jmath_!$ and $\jmath_*$;
\item\label{gluing dts} there exist natural transformations $\imath_*\imath^* \to \jmath_!\jmath^*[1]$ and $\jmath_*\jmath^* \to \imath_*\imath^![1]$ such that there are natural exact triangles
\[
\imath_*\imath^! \to \id \to \jmath_*\jmath^* \to \imath_*\imath^! [1] 
\quad \textrm{and} \quad 
\jmath_!\jmath^* \to \id \to \imath_*\imath^* \to \jmath_!\jmath^* [1]
\]
{ whose other morphisms are units or counits of the relevant adjunctions; }
\item { the units of the adjunctions $\imath_* \dashv \imath^!$ and $\jmath_! \dashv \jmath^*$ are isomorphisms, as are the counits of $\imath^* \dashv \imath_*$ and $\jmath^* \dashv \jmath_*$.}
\end{enumerate}
Some of these conditions are redundant:  the existence of any one of the adjoints guarantees the existence of the other three, the triangles in the third condition are dual to one another and the final condition follows from the fact that $\imath_*, \jmath_*$ and $\jmath_!$ are fully faithful.

Under these conditions \cite[Th\'eor\`eme 1.4.10]{bbd} states that one can glue given $t$-structures $\cat{A}^{\leq0}\subset\cat{A}$ and $\cat{C}^{\leq0}\subset\cat{C}$ to obtain a $t$-structure
\[
\cat{B}^{\leq 0} = \langle b\in \cat{B} \ | \ \imath^*b \in \cat{A}^{\leq 0}, \jmath^*b \in \cat{C}^{\leq 0} \rangle
\]
on $\cat{B}$. This glued $t$-structure is self-dual whenever the given ones on $\cat{A}$ and $\cat{C}$ are so. With respect to these $t$-structures
\begin{enumerate}
\item $\imath_*$ and $\jmath^*$ are $t$-exact;
\item $\imath^!$ and $\jmath_*$ are left $t$-exact;
\item $\imath^*$ and $\jmath_!$ are right $t$-exact.
\end{enumerate}
The adjunctions give rise to natural morphisms $\jmath_!\to\jmath_*$ and $\imath^! \to \imath^*$. The \defn{intermediate extension}  is defined to be the functor $\jmath_{!*}=\im (\pf{\jmath_!} \to \pf{\jmath_*})$ and similarly the \defn{intermediate restriction} is defined to be $\imath^{!*}=\im (\pf{\imath^!} \to \pf{\imath^*})$. By construction both $\jmath_{!*}$ and $\imath^{!*}$ commute with duality.  

If $\jmath$ and $k$ are composable quotient functors  then $(\jmath k)_{!*} = \jmath_{!*}k_{!*}$, see \cite[2.1.7.1]{bbd}; the analogue for composable inclusion functors is false in general, see Example \ref{cs2 failure}. Intermediate extensions are neither left nor right exact, but do preserve injections, surjections and images; intermediate restrictions need not have any exactness properties. Finally, intermediate extensions are fully faithful.

\begin{examples}
\label{triangulated gluing examples}
There are various examples of this gluing situation:
\begin{enumerate}
\item The bounded derived category of a highest weight category $\cat{B}^0$ with duality in the sense of \cite{MR982273}. The simple objects of $\cat{B}^0$ are the elements of a poset; each is fixed by duality. The functor $\imath_*$ is given by the inclusion of $D^b(\cat{A}^0)$  where $\cat{A}^0$ is the Serre subcategory generated by a downward-closed subset of simple objects. The functor $\jmath^*$ is the induced quotient $D^b(\cat{B}^0) \to D^b(\cat{B}^0/\cat{A}^0)$. 
\item The constructible derived category of sheaves of vector spaces on a finite-dimensional topologically stratified space, with only even-dimensional strata, equipped with Verdier duality and the self-dual perverse $t$-structure. In this case $\imath$ is the inclusion of a closed union of strata and $\jmath$ the complementary inclusion of an open union of strata, and $\imath_*$ and $\jmath^*$ the respective induced functors.

The example of the constructible derived category of sheaves of torsion modules over a Dedekind ring $R$ works similarly. 
\item Let $g\colon X \to Z$ and $h \colon Z \to Y$ be a proper birational morphisms of smooth complex algebraic surfaces, and let $f=g\circ h$. Then the exact triple of null categories, as defined in Examples \ref{triangulated cats with duality} (\ref{flops}) above,
\[
C_g \stackrel{\imath_*}{\longrightarrow} C_f \stackrel{Rg_*}{\longrightarrow} C_h
\]
where the first functor is the inclusion of the full subcategory $C_g$ into $C_f$, extends to gluing data by \cite[Proposition 3.5]{bodzenta-canonical}. This data is compatible as above with shifted Grothendieck duality since we have restricted to surfaces.
\end{enumerate}
\end{examples}

\begin{proposition}
\label{direct sum decomp}
{ Consider as before a gluing $\cat{A} \stackrel{\imath_*}{\longrightarrow} \cat{B} \stackrel{\jmath^*}{\longrightarrow} \cat{C}$ of self-dual 
$t$-structures, with induced dualities on their hearts $\cat{A}^0, \cat{B}^0$ and $\cat{C}^0$.}
Suppose $\cat{B}^0$ is noetherian, equivalently that both $\cat{A}^0$ and $\cat{C}^0$ are noetherian. Then $\wgc{}{\cat{B}^0} \cong  \wgc{}{\cat{A}^0} \oplus \wgc{}{\cat{C}^0}$. The analogue for anti-symmetric forms also holds.
\end{proposition}
\begin{proof}
By \cite[Proposition 1.4.26]{bbd} each simple object of the heart $\cat{B}^0$ is either of the form $\imath_*a$  for simple $a\in \cat{A}^0$, or $\jmath_{!*}c$ for simple $c\in \cat{C}^0$. Furthermore duality preserves these two classes. Hence by the same argument as in the proof of Corollary \ref{direct-sum-cor} we have
\begin{align*}
\wgc{}{\cat{B}^0} &\cong \wgc{}{\langle \imath_*a \ | \ \textrm{simple}\ a\in\cat{A}^0 \rangle }
 \oplus \wgc{}{\langle \jmath_{!*}c \ | \ \textrm{simple}\ c\in\cat{C}^0 \rangle }\:.
\end{align*}
It is immediate that $\imath_*$ induces an isomorphism $\wgc{}{\cat{A}^0} \cong \wgc{}{\langle \imath_*a \ | \ \textrm{simple}\ a\in\cat{A}^0 \rangle }$. It follows from the fact that $\jmath_{!*}$ is fully faithful that $\jmath^*$ induces an isomorphism
\[
\wgc{}{\langle \jmath_{!*}c \ | \ \textrm{simple}\ c\in\cat{C}^0 \rangle } \cong \wgc{}{\cat{C}^0}. \qedhere
\]
\end{proof}
Together with Proposition \ref{abelian-triangulated relation} this provides an independent proof of the existence of the long exact sequence of Balmer--Witt groups in the case when $\cat{B}^0$ is noetherian, and furthermore shows that it splits in this case.

The proof shows that the inclusion $\wgc{}{\cat{A}^0} \mono \wgc{}{\cat{B}^0}$ is induced by $\imath_*$ and the projection $\wgc{}{\cat{B}^0} \epi \wgc{}{\cat{C}^0}$ by $\jmath^*$. It is harder to obtain explicit descriptions of the other inclusion and projection.
 
\begin{theorem}
\label{sum thm}
Suppose $\beta \colon  b \to Db$ is a non-degenerate symmetric form in $\cat{B}$. Then 
\begin{equation}
\label{witt decomposition}
[\beta] = [\imath_*\imath^{!*}\beta] + [\jmath_{!*}\jmath^*\beta].
\end{equation}
in the Witt group {  $\wgc{}{\cat{B}^0}\cong \wgc{}{\cat{B}}$}. 
\end{theorem}
\begin{proof}
 There is an exact triangle $\imath_*\imath^! b \to b \to \jmath_*\jmath^* b\to \imath_*\imath^! b[1]$
which gives rise to a long exact sequence
\[
\cdots \to \imath_*\pf{\imath^!}b \to H^0b \to \pf{\jmath_*}\jmath^* b \to \cdots
\]
in the heart $\cat{B}^0$. We apply Corollary \ref{splitting cor} to write $[H^0\beta]=[\beta]$ as a sum of two terms in the Witt group $\wgc{}{\cat{B}^0} \cong \wgc{}{\cat{B}}$. The requisite two terms are the induced forms on the images of $\alpha$ and $\gamma$ in the diagram 
\[
\begin{tikzcd}
{} & \imath_*\pf{\imath^!}b \ar{r}\ar{d}{\alpha} & H^0b \ar{d}{\beta} & \pf{\jmath_!}\jmath^* Db \ar{l} \ar{d}{\gamma} &\\
\imath_*\pf{\imath^*}Db \ar[equal]{r} & D\imath_*\pf{\imath^!}b & DH^0b \ar{l} \ar{r} & D\pf{\jmath_!}\jmath^*Db \ar[equal]{r} & \pf{\jmath_*}\jmath^*D^2b.
\end{tikzcd}
\]
To identify $\alpha$ consider the commutative diagram below in which $\alpha$ is the composite of the dashed arrows.
\[
\begin{tikzcd}
DH^0b \ar[dashed]{rrr} &&&\imath_*\pf{\imath^*}Db \\
& H^0b \ar{r} \ar[dashed]{ul}{\beta} & \imath_*\pf{\imath^*}b\ar{ur}[swap]{\imath_*\pf{\imath^*}\beta}&\\
&\imath_*\pf{\imath^!}b \ar[dashed]{u} \ar{dl}[swap]{\imath_*\pf{\imath^!}\beta} \ar[two heads]{r} & \imath_*\imath^{!*}b\ar{dr}{ \imath_*\imath^{!*}\beta}\ar[hook]{u}&\\
\imath_*\pf{\imath^!}Db \ar[two heads]{rrr} \ar{uuu} &&& \imath_*\imath^{!*}Db \ar[hook]{uuu}
\end{tikzcd}
\]
The top and left squares arise from the natural morphisms $\id \to \imath_*\pf{\imath^*}$ and $\imath_*\pf{\imath^!}\to \id$ respectively, and the bottom and right squares from the definition of $\imath^{!*}$. The central square commutes because the composite $\imath_*\pf{\imath^!}\to \id \to \imath_*\pf{\imath^*}$ is the natural morphism $\imath_*(\pf{\imath^!}\to \pf{\imath^*})$, see \cite[1.4.21.1]{bbd}. It follows that $\overline{\alpha}\cong\imath_*\imath^{!*}\beta$. A similar, slightly more involved, argument shows that $\overline{\gamma}\cong\jmath_{!*}\jmath^*\beta$. 
\end{proof}

{
\begin{example}
\label{abstract-nonsingular}
Assume the  non-degenerate symmetric form $\beta \colon  b \to Db$ in $\cat{B}^0$ is a direct sum $\beta= \imath_*\alpha \oplus \jmath_{!*}\gamma$,
with non-degenerate symmetric forms $\alpha \colon  a \to Da$ in $\cat{A}^0$ and $\gamma \colon  c \to Dc$ in $\cat{C}^0$.
Then $\imath^{!*}\beta\cong \alpha$ and $\jmath^*\beta \cong \gamma$ so that (\ref{witt decomposition}) is the image of this decomposition of $\beta$ in the Witt group.
\end{example}}

It is natural to assume that, when $\cat{B}^0$ is noetherian, { the sum in (\ref{witt decomposition})} corresponds to the direct sum  decomposition of Proposition \ref{direct sum decomp}. This is false in general. The individual terms depend upon the choice of representative $\beta$ not just on the class $[\beta]$ --- see Example \ref{ie not exact 1}. When $\jmath_{!*}$ is exact then it induces a map of Witt groups splitting $\jmath^*$, and moreover $\jmath_{!*}\gamma$ is Witt equivalent to a sum of forms on intermediate extensions of simple objects in $\cat{C}^0$. Since $\imath_*$ is a monomorphism it follows that the class $[\imath^{!*}\beta]$ is also well-defined, and so $\imath^{!*}$ induces a map of Witt groups too, splitting $\imath_*$. To summarise:
\begin{corollary}
\label{exact extension cor}
Suppose $\cat{B}^0$ is noetherian and the intermediate extension $\jmath_{!*}$ is exact. Then the direct sum decomposition of Proposition \ref{direct sum decomp} is given by the maps $[\beta] \mapsto \left( [\imath^{!*}\beta],\jmath^*[\beta]\right)$ and $([\alpha],[\gamma]) \mapsto \imath_*[\alpha] + [\jmath_{!*}\gamma]$.
\end{corollary}
Another important case is when the form $\beta$ is anisotropic.
\begin{corollary}
\label{anisotropic splitting}
Suppose $\cat{B}^0$ is noetherian and $\beta \colon  b \to Db$ is a non-degenerate anisotropic symmetric form in $\cat{B}^0$. Then there is an isometry $\beta \cong \imath_*\imath^{!*}\beta \oplus \jmath_{!*}\jmath^*\beta$ and (\ref{witt decomposition}) corresponds to the direct sum decomposition of Proposition \ref{direct sum decomp}. 
\end{corollary}
\begin{proof}
The existence of the isometry follows from Remarks \ref{splitting remarks}. As a consequence $\imath_*\imath^{!*}\beta$ and $\jmath_{!*}\jmath^*\beta$ are also anisotropic. Hence each is a direct sum of non-degenerate forms on simple objects. It is clear that $\imath_*\imath^{!*}b$ has only factors of the form $\imath_*a$ for simple $a\in \cat{A}^0$. Since the intermediate extension $\jmath_{!*}\jmath^*b$ cannot have subobjects of the form $\imath_*a$ it follows that no such objects can appear when we write it as a direct sum of simple objects. Hence $\imath_*\imath^{!*}b \in \langle \imath_*a \ | \textrm{simple}\ a\in\cat{A}^0 \rangle$ and $\jmath_{!*}\jmath^*b \in \langle \jmath_{!*}c \ | \textrm{simple}\ c\in\cat{C}^0 \rangle$. The result follows.
\end{proof}
In particular, it follows that the classes $[\imath^{!*}\beta']$ and $[\jmath_{!*}\gamma']$ are well-defined independent of the choice of {\em anisotropic} representatives $\beta'$ for $[\beta]$ and $\gamma'$ for $[\gamma]$. Thus we can define homomorphisms
\[
\imath^{!*} \colon  \wgc{}{\cat{B}^0} \to  \wgc{}{\cat{A}^0} \colon  [\beta] \mapsto [\imath^{!*}\beta']
\]
and $\jmath_{!*} \colon  \wgc{}{\cat{C}^0} \to  \wgc{}{\cat{B}^0} \colon  [\gamma] \mapsto [\jmath_{!*}\gamma']$ where $\beta'$ and $\gamma'$ are (choices of) anisotropic representatives. The projections and inclusions of the direct sum decomposition are then the homomorphisms
\begin{equation}
\label{inclusions and projections}
\begin{tikzcd}
\wgc{}{\cat{A}^0} \ar[bend left=10]{r}{\imath_*} 
& \wgc{}{\cat{B}^0} \ar[bend left=10]{l}{\imath^{!*}} \ar[bend left=10]{r}{\jmath^*}
& \wgc{}{\cat{C}^0} \ar[bend left=10]{l}{\jmath_{!*}}.
\end{tikzcd}
\end{equation}
In practice it may be difficult to identify maximal isotropic subobjects in $\cat{B}^0$, but easier to do so in $\cat{C}^0$, for instance in the next section $\cat{B}^0$ will be a category of perverse sheaves and $\cat{C}^0$ a category of local systems on a stratum. The following approach allows one to compute the canonical direct sum decomposition of Proposition \ref{direct sum decomp} provided one can find maximal isotropic subobjects in $\cat{C}^0$. Let $c \mono \jmath^*b$ be a maximal isotropic subobject of $\jmath^*\beta$. Then $\jmath_{!*}c \mono \jmath_{!*}\jmath^*b$ is isotropic for $\jmath_{!*}\jmath^*\beta$. Let $\beta' = \ir{\jmath_{!*}\jmath^*\beta}{\jmath_{!*}c} $ be the reduction. Apply Theorem \ref{sum thm} to $\beta'$ to obtain
\[
[\jmath_{!*}\jmath^*\beta] = [\beta'] =[\imath_*\imath^{!*}\beta'] + [\jmath_{!*}\jmath^*\beta']
\] 
and note that $\jmath^*\beta' = \ir{\jmath^*\beta}{c}$ is anisotropic. It follows that this is the canonical decomposition of $[\jmath_{!*}\jmath^*\beta]$. Hence the canonical decomposition of $[\beta]$ is 
\begin{equation}
\label{explicit sum}
[\beta] = \left( [\imath_*\imath^{!*}\beta] + [\imath_*\imath^{!*}\beta']\right) + [\jmath_{!*}\jmath^*\beta'].
\end{equation}

It is clear that $\imath_*$ and $\jmath^*$ preserve anisotropy; the same holds for intermediate extension and restriction:
\begin{lemma}
\label{anisotropy preserved}
If $\beta\colon b \to Db$ and $\gamma\colon c\to Dc$ are anisotropic symmetric forms in  $\cat{B}$ and $\cat{C}$ respectively then $\imath^{!*}\beta$ and $\jmath_{!*}\gamma$ are also anisotropic.
\end{lemma}
\begin{proof}
For intermediate restrictions this follows from \cite[Proposition 1.4.17]{bbd} and the fact that $\imath_*$ is $t$-exact, which together imply that $\imath_*{}^p\imath^!b \to b$ is a monomorphism. For intermediate extension we note that if $c' \mono \jmath_{!*}c$ is an isotropic subobject then $\jmath^*c'=0$, otherwise it would be an isotropic subobject of $c$. Hence $c' \cong \imath_*\imath^!c'$. But this is impossible unless $c'=0$ as intermediate extensions cannot have non-zero subobjects of this form \cite[Corollaire 1.4.25]{bbd}.
\end{proof}

\begin{remark}
The results of this subsection also hold, with essentially the same proofs, in the context of gluing of abelian categories in the sense of \cite{MR2054979}. In this context one has the same six functor formalism, but with exact sequences 
\[
0\to \imath_*\imath^! \to \id \to \jmath_*\jmath^* 
\quad \textrm{and} \quad 
\jmath_!\jmath^* \to \id \to \imath_*\imath^* \to 0
\]
replacing the corresponding exact triangles, see \cite[Proposition 4.2]{MR2054979}. As above, the simple objects of the glued abelian category have either the form $\imath_*a$ or $\jmath_{!*}c$, see \cite[Lemma 2]{MR2422265}. Since we only use exactness in the middle in the proof of Theorem \ref{sum thm} everything works as before. 

This is a more general context; there are abelian gluing examples which do not come from gluing of triangulated categories. In particular Examples \ref{triangulated gluing examples} (1) and (3) can be generalised, see \cite[Lemma 2.5]{MR3742477} and \cite[Proposition 3.11]{bodzenta-canonical} respectively.
\end{remark}

\section{Application to stratified spaces}
\label{stratified spaces}

\subsection{Witt groups of local systems}

Let $X$ be a locally-connected topological space, and let $\loc{X}$ be the category of local systems on $X$ with coefficients in a field $\F$. { When $X$ is connected} this category is  equivalent to the category of $\F$-representations of the fundamental group $\pi_1 X$. A representation $\rho \colon \pi_1X \to \mathrm{GL}(V)$ has a dual representation on the vector space dual $V^*$ given by $g \mapsto \rho(g)^{-*}{ :=\rho(g^{-1})^*}$. There is an induced duality on local systems which we denote by $\mathcal{L} \mapsto \mathcal{L}^\vee$. Let $W(\loc{X})$ be the associated Witt group. It is a ring under the tensor product of local systems, and is covariantly functorial under continuous maps { (see also~\cite{Bunke-Ma}). If $X$ is a topological manifold, then there is also
a second duality  $\mathcal{L} \mapsto \mathcal{L}^\vee\otimes or_X$ obtained by in addition twisting  with the orientation sheaf $or_X$ of $X$. Let $W(\loc{X},or_X)$ be the associated Witt group, which agrees with  $W(\loc{X})$ when $X$ is oriented, i.e.\ when an isomorphism $or_X\cong \F_X$ has been chosen.}

\subsection{Witt groups of perverse sheaves}
\label{perverse sheaves}

Let $X$ be a finite-dimensional topologically stratified space, \ie a locally cone-like stratified space.  Let $\constr{X}$ be the bounded derived category of constructible sheaves of $\F$-vector spaces on $X$ for a field $\F$. The Poincar\'e--Verdier dual $D$ makes this into a category with duality. Suppose that $X$ has only  even-dimensional strata. Then there is a self-dual perversity $p(S) = -\dim S / 2$ which defines a $t$-structure $\pf{D}^{\leq 0}(X)$ on $\constr{X}$ whose heart is the category $\perv{X}$ of perverse sheaves. This is the full subcategory of $\constr{X}$ whose objects obey the vanishing conditions
\[
\shcoh{j}{k_S^*\mathcal{A}} = 0 \ \textrm{for}\  j > - \dim{S} /2 \qquad \textrm{and} \qquad
\shcoh{j}{k_S^!\mathcal{A}} = 0  \ \textrm{for}\  j < -\dim{S} / 2
\]
where $k_S \colon  S \hookrightarrow X$ is the inclusion of a stratum. The category $\perv{X}$ is both artinian and noetherian.

It follows from the fact that the above vanishing conditions are local on $X$ that tensoring with a local system $\mathcal{L}$ is an exact functor 
\[
- \otimes \mathcal{L} \colon \perv{X} \to \perv{X}.
\]
Moreover, Verdier duality and the duality on local systems are related by 
\[
D\left( \mathcal{A} \otimes \mathcal{L} \right) \cong D\mathcal{A} \otimes \mathcal{L}^\vee
\]
 where $\mathcal{A}$ is a perverse sheaf and $\mathcal{L}$ a local system. Combining these facts we obtain:
\begin{lemma}
Tensor product makes the Witt group $W(\perv{X})$ of perverse sheaves into a module over the Witt group $W(\loc{X})$ of local systems.
\end{lemma}

Let $\imath \colon  Y \hookrightarrow X$ be the inclusion of a closed stratified subspace, in other words $Y$ is a closed union of strata of $X$. Let $\jmath \colon  U=X-Y \hookrightarrow X$ be the complementary open inclusion.  Then 
\[
\constr{Y} \stackrel{\imath_*}{\longrightarrow} \constr{X} \stackrel{\jmath_*}{\longrightarrow} \constr{U} 
\]
is an exact triple of triangulated categories satisfying the conditions of \S\ref{gluing and splitting}. The perverse $t$-structure on $\constr{X}$ is glued from the perverse $t$-structures on $\constr{Y}$ and $\constr{U}$. For the remainder of this section we assume that the stratified space $X$ has only finitely many strata, which is the case, for instance, if it is compact.  For ease of reading we suppress extensions by zero from closed unions of strata. 

\begin{corollary}
\label{witt decomposition theorem}
There is a direct sum decomposition
\begin{equation}
\label{witt group sum}
\wgc{}{\perv{X}} \cong \bigoplus_{S\subset X} { \wgc{\epsilon_S}{\loc{S},or_S} }
\end{equation}
where $\epsilon_S=(-1)^{\dim S /2}$.
\end{corollary}
\begin{proof}
The decomposition is obtained by applying Proposition \ref{direct sum decomp} repeatedly to obtain
\[
\wgc{}{\perv{X}} \cong \bigoplus_{S\subset X} \wgc{}{\perv{S}}.
\]
The decomposition in the statement is equivalent: a perverse sheaf on $S$ is a local system shifted in degree by $\dim S/2$ and this accounts for the signs $\epsilon_S$ because odd shifts switch symmetric and antisymmetric forms (see \cite[Remark 2.16]{twg1}). { Moreover, Verdier duality corresponds under this identification to the duality of local systems twisted by the orientation sheaf $or_S$ of the stratum $S$.}
\end{proof}

{
\begin{example}
\label{ex:Bunke-Ma}
Assume all strata $S$ are orientable and consider the coefficient field $\F=\R$. Then 
$ \wgc{\epsilon_S}{\loc{S},or_S}\cong  \wgc{\epsilon_S}{\loc{S}}$ is by~\cite{Bunke-Ma} a direct sum of a free $\Z$-module and  a torsion module whose elements are all of order two. So the same is true for the Witt group $\wgc{}{\perv{X}}$ of perverse sheaves.
\end{example}}

We now discuss how to compute the associated `canonical decomposition' of a class in $\wgc{}{\perv{X}}$  into classes of forms on local systems on the strata, or equivalently on their intermediate extensions. Let $\imath_S \colon \overline{S} \hookrightarrow X$ and $\jmath_S \colon S \hookrightarrow \overline{S}$ be the inclusions, so that $k_S = \imath_S \circ \jmath_S$. Let $\beta|_S$ be the restricted form
\[
\jmath_S^*\pf{\imath_S}^{!}\mathcal{B} \to { \jmath_S^*\pf{\imath_S}^{!}D\mathcal{B} \cong 
\jmath_S^*D\left(\pf{\imath_S}^{*}\mathcal{B}\right) }
\]
induced by a symmetric form $\beta \colon \mathcal{B} \to D\mathcal{B}$. This restricted form may be degenerate; the associated non-degenerate form is, by definition, $\jmath_S^*\imath_S^{!*}\beta$.

 \begin{lemma}
Suppose $\beta\colon \mathcal{B}\to D\mathcal{B}$ is non-degenerate and anisotropic. Then there is an isometry
\begin{equation}
\label{anisotropic isometry}
\beta \cong \sum_{S \subset X} {\jmath_S}_{!*}\jmath_S^* \imath_S^{!*} \beta,
\end{equation}
and passing to the Witt group we obtain the canonical decomposition of $[\beta]$. 
\end{lemma}
\begin{proof}
The existence of the isometry, and the fact that it corresponds to the direct sum decomposition, follow from Corollary \ref{anisotropic splitting}: applying it first to $\imath_S$ and the complementary open inclusion, and then to $\jmath_S$ and the complementary closed inclusion yields an isometry
\[
\beta \cong \beta' \oplus {\jmath_S}_{!*}\jmath_S^* \imath_S^{!*}\beta \oplus \beta''
\]
where the middle term is the summand associated to the stratum $S$.
\end{proof}
The next lemma reduces the problem of identifying an anisotropic form on a perverse sheaf to the analogous question for local systems.
\begin{lemma}
\label{anisotropy by restrictions}
A symmetric form $\beta \colon \mathcal{B} \to D\mathcal{B}$ is anisotropic if, and only if, for each stratum $S$ the restriction $\beta|_S$ is anisotropic.
\end{lemma}
\begin{proof}
Suppose $\mathcal{A}\hookrightarrow \mathcal{B}$ is a non-zero isotropic subobject for $\beta$. Let $S$ be a maximal stratum for which $\mathcal{A}|_S\neq 0$. Then $k_S^*\mathcal{A} = \jmath_S^*\pf{\imath_S^!}\mathcal{A}  \hookrightarrow \jmath_S^*\pf{\imath_S^!}\mathcal{B}$ is a non-zero isotropic subobject for the restriction $\beta|_S$. 

In the other direction, if $\mathcal{C} \hookrightarrow \jmath_S^*\pf{\imath_S^!}\mathcal{B}$ is a non-zero isotropic subobject for $\beta|_S$ then the image of the composite
\[
\pf{\jmath_S}_! \mathcal{C} \to \pf{\imath_S^!}\mathcal{B} \to \mathcal{B}
\]
is a non-zero isotropic subobject for $\beta$.
\end{proof}
We now describe an inductive procedure for computing the canonical decomposition of a general class in $W(\perv{X})$. In order to do so we extend the partial order $S \leq T \iff S \subset \overline{T}$ on the strata of $X$ to a total order, and label the strata so that $S_1 > \ldots > S_n$. For $1\leq k<n$ let 
\[
\imath_k \colon S_{k+1}\cup\cdots\cup S_n \hookrightarrow S_k\cup\cdots\cup S_n 
\]
be the closed inclusion, and for $1\leq k\leq n$ let 
\[
\jmath_k \colon S_k \hookrightarrow S_k\cup\cdots\cup S_n
\]
 be the (complementary) open inclusion, in particular with $\jmath_n \colon S_n \to S_n$ the identity. {Let $\tilde \imath_k = \imath_1\imath_2\cdots \imath_k \colon S_{k+1}\cup\cdots\cup S_n \hookrightarrow X$ be the composite.}
\begin{lemma}
\label{anisotropy in stages}
Suppose $\beta \colon \mathcal{B} \to D\mathcal{B}$ is a non-degenerate symmetric form in $\perv{X}$ such that $\beta|_{S_1\cup\ldots \cup S_{k-1}}$ is anisotropic. Then $\beta$ has an isotropic subobject such that the reduction by it, say $\beta' \colon \mathcal{B}' \to D\mathcal{B}'$, satisfies
\begin{enumerate}
\item $\beta'|_{S_1 \cup \cdots \cup S_{k-1}} = \beta|_{S_1 \cup \cdots \cup S_{k-1}}$;
\item $\beta'|_{S_k}$ is the reduction of $\beta|_{S_k}$ by a maximal isotropic subobject.
\end{enumerate}
Note that  Lemma \ref{anisotropy by restrictions} then implies that $\beta'|_{S_1 \cup \cdots \cup S_{k}}$ is anisotropic.
\end{lemma}
\begin{proof}
Let { $\mathcal{A} \hookrightarrow \jmath_k^* \pf{\,\tilde \imath_{k-1}^{\,!}}\mathcal{B}$ }be a maximal isotropic subobject for $\beta|_{S_k}$. Then the image of the composite { $\pf{\jmath_{k!}}\mathcal{A} \to \pf{\,\tilde \imath_{k-1}^{\,!}}\mathcal{B}  \to \mathcal{B}$} is isotropic for $\beta$. Let $\beta'$ be the reduction. Since { $\pf{\jmath_{k!}}\mathcal{A}$} is supported on $S_k\cup \cdots \cup S_n$ the first condition
\[
\beta'|_{S_1 \cup \cdots \cup S_{k-1}} = \beta|_{S_1 \cup \cdots \cup S_{k-1}}
\]
is satisfied. 

By construction $\beta'|_{S_k}$ is anisotropic. Since $\mathcal{A}$ was chosen to be a maximal isotropic subobject $\beta'|_{S_k}$ is isometric to the reduction of $\beta|_{S_k}$ by $\mathcal{A}$. 
\end{proof}
The procedure for constructing an anisotropic representative, and for computing the canonical decomposition is as follows. Set  $\beta_0=\beta$. Using Lemma \ref{anisotropy in stages} we construct, by successive isotropic reductions, forms $\beta_1, \ldots,\beta_n$ such that 
\begin{enumerate}
\item $\beta_k |_{S_1\cup\cdots\cup S_k}$ is anisotropic;
\item $\beta_k|_{S_1 \cup \cdots \cup S_{k-1}} = \beta_{k-1}|_{S_1 \cup \cdots \cup S_{k-1}}$;
\item $\beta_k|_{S_k}$ is the reduction of $\beta_{k-1}|_{S_k}$ by a maximal isotropic subobject.
\end{enumerate}
In particular $\beta_n$ is an { anisotropic} representative for $[\beta]$, and the canonical decomposition is
\begin{equation}
\label{direct sum formula}
[\beta] = \sum_{k=1}^n [{\jmath_k}_{!*} \left( \beta_k|_{S_k}\right)].
\end{equation}
The (anti)symmetric local systems of (\ref{witt group sum}) are obtained from the $\beta_k|_{S_k}$ by shifting by $\dim S_k/2$. We now investigate circumstances in which it is possible to find  explicit expressions for the $\beta_k|_{S_k}$  in terms of $\beta$. 

Applying Theorem \ref{sum thm} inductively, starting with the complementary inclusions $(\imath_1,\jmath_1)$, one obtains a formula
\begin{equation}
\label{cs1}
[\beta] = [ {\jmath_1}_{!*}\jmath_1^*\beta] + \sum_{k=2}^{n} [{\jmath_{k}}_{!*}\jmath_{k}^*\imath_{k-1}^{!*}\cdots\imath_{1}^{!*}\beta].
\end{equation}
In general this is not the above canonical decomposition. There are many similar formul\ae, corresponding to different ways off splitting off strata. These formul\ae\  may differ from one another, and each may depend on $\beta$ not merely its class $[\beta]$. 

When $\beta$ is anisotropic Corollary \ref{anisotropic splitting} guarantees that (\ref{cs1}) is the canonical decomposition, and so must agree with (\ref{anisotropic isometry}). In fact we can verify that the given representatives of terms in (\ref{cs1}) are isometric to those in (\ref{anisotropic isometry}) not merely Witt-equivalent.
\begin{proposition}
\label{anisotropic formula}
Suppose $\beta\colon \mathcal{B}\to D\mathcal{B}$ is non-degenerate and anisotropic. Then there are isometries 
\[
{\jmath_{k}}_{!*}\jmath_{k}^*\imath_{k-1}^{!*}\cdots\imath_{1}^{!*}\beta \cong {\jmath_{S_k}}_{!*}\jmath_{S_k}^*\imath_{S_k}^{!*} \beta
\]
for each $k=2,\ldots,n$ so that  (\ref{cs1}) is the image
\begin{equation}
\label{cs2 order independent}
[\beta] = \sum_{S \subset X} [ {\jmath_S}_{!*}\jmath_S^* \imath_S^{!*} \beta ]
\end{equation}
of the isometry (\ref{anisotropic isometry}) in the Witt group.
\end{proposition}
\begin{proof}
Let $\imath \colon Y \hookrightarrow X$ be the inclusion of a closed union of strata. Then it follows from \cite[Proposition 1.4.17]{bbd} that the (dual) natural morphisms $\imath_*{}^p\imath^!\mathcal{B} \to \mathcal{B}$ and $D\mathcal{B} \to \imath_*{}^p\imath^*D\mathcal{B}$ are respectively monomorphic and epimorphic. Hence there is a commutative diagram
\[
\begin{tikzcd}
\ker \alpha \ar[hook]{r} \ar{d}{0} & \imath_*{}^p\imath^!\mathcal{B}  \ar[hook]{r} \ar{d}{\alpha} & \mathcal{B} \ar{d}{\beta} \\
D(\ker \alpha)  & \imath_*{}^p\imath^*D\mathcal{B} \ar[two heads]{l}  & D\mathcal{B} \ar[two heads]{l}
\end{tikzcd}
\]
where $\alpha$ is the restriction of $\beta$. As $\beta$ is anisotropic we deduce that $\ker \alpha =0$, and hence also $\coker \alpha \cong D(\ker D\alpha)\cong D(\ker \alpha) =0$. Therefore $\alpha$ is an isomorphism and ${}^p\imath^!\mathcal{B} \cong \imath^{!*}\mathcal{B} \cong \pf{\imath^*}\mathcal{B}$. By \cite[Proposition 1.3.17]{bbd} ${}^p\imath_r^{!}{}^p\imath_{r-1}^! \cong {}^p(\imath_{r-1} \imath_r)^!$ and similarly ${}^p\imath_r^*{}^p\imath_{r-1}^* \cong {}^p(\imath_{r-1} \imath_r)^*$. Combining these we see that $\imath_r^{!*}\imath_{r-1}^{!*}\beta \cong (\imath_{r-1}  \imath_r)^{!*}\beta$ (as forms not merely as Witt classes). By induction $\imath_r^{!*}\cdots \imath_1^{!*}\beta = (\imath_1 \cdots \imath_r)^{!*}\beta$. 

One can then check that ${\jmath_{k}}_{!*}\jmath_{k}^*(\imath_1 \cdots   \imath_{k-1})^{!*} \cong {\jmath_{S_k}}_{!*}\jmath_{S_k}^*\imath_{S_k}^{!*}$ are naturally isomorphic, so (\ref{cs1}) becomes (\ref{cs2 order independent}). \end{proof}
For applications it is more useful to identify geometric conditions under which (\ref{cs1}) is the canonical decomposition, and hence is independent of the representative $\beta$ and choice of ordering of the strata. We approach this by identifying conditions under which intermediate extensions are exact, and then using Corollary (\ref{exact extension cor}).
\begin{lemma}
\label{global exact condition}
Suppose that $S$ is a stratum with finite fundamental group, and that the characteristic of $\F$ does not divide the order of $\pi_1S$. Then the intermediate extension ${\jmath_S}_{!*}$ is exact.
\end{lemma}
\begin{proof}
If $\pi_1S$ is finite then $\perv{S}$ is semi-simple by Maschke's Theorem. The result follows because ${\jmath_S}_{!*}$ is additive. 
\end{proof}
\begin{lemma}
\label{local exact condition}
Let  $S > T$ be strata in $X$, and let $L$ be the link of $T$ in $\overline{S}$. Suppose that the intersection cohomology group $\ih{(\dim L-1)/2}{L;\mathcal{L}}=0$ for any local system $\mathcal{L}$ on the link.
Then
\[
\ext{1}{\mathcal{A}}{\mathcal{B}} = 0 = \ext{1}{\mathcal{B}}{\mathcal{A}}
\]
where $\mathcal{A}={\jmath_S}_{!*}\mathcal{L}[\dim S /2]$ and $\mathcal{B}={\jmath_T}_{!*}\mathcal{M}[\dim T/2]$ are the intermediate extensions of (shifted) local systems respectively on $S$ and on $T$.  In fact this holds if the above intersection cohomology group vanishes for those local systems which arise as the restriction of a local system on $S$.
\end{lemma}
\begin{proof}
It suffices to prove that $\ext{1}{\mathcal{A}}{\mathcal{B}} = 0$ for any such $\mathcal{A}$ and $\mathcal{B}$, since then by duality $\ext{1}{\mathcal{B}}{\mathcal{A}}\cong \ext{1}{D\mathcal{A}}{D\mathcal{B}} =0$. By adjunction and the fact that $\mathcal{B} \in \pf{D}^{0}(\overline{T})$ and ${\imath_T}^*\mathcal{A} \in \pf{D}^{< 0}(\overline{T})$ we have
\[
\ext{1}{\mathcal{A}}{\mathcal{B}} \cong \ext{1}{{\imath_T}^*\mathcal{A}}{\mathcal{B}} \cong \mor{H^{-1}({\imath_T}^*\mathcal{A})}{\mathcal{B}}
\]
where $H^{-1}$ is cohomology with respect to the standard, not the perverse, $t$-structure. Since $\mathcal{B}$ has no subobjects supported on $\overline{T}-T$ the right hand group vanishes, for any such $\mathcal{B}$, if $H^{-1}(\imath_T^*\mathcal{A})$ is supported on $\overline{T}-T$. This is equivalent to the vanishing of the stalk of $H^{-1}(\imath_T^*\mathcal{A})$ at some, hence at all, $x\in T$. This stalk is $\ih{(\dim L-1)/2}{L;\mathcal{L}|_L}$. The result follows.
\end{proof}
The conditions of this lemma are satisfied if, for instance, $X$ is Whitney stratified, all strata have smooth closures --- so that all links of pairs of strata are spheres --- and all such links have dimension $\geq 3$. In particular it holds for subspace arrangements where the dimension of pairs of subspaces differ by at least $3$.
\begin{corollary}
\label{int ext exact}
Suppose that for each pair of strata $S > T$ in $X$ and local system $\mathcal{L}$ on the link $L$ of $T$ in $\overline{S}$  the intersection cohomology group  $\ih{(\dim L-1)/2}{L;\mathcal{L}}=0$. Let $\jmath \colon Y \hookrightarrow \overline{Y}$ be the inclusion of a locally-closed union of strata $Y$ in its closure. Then the intermediate extension $\jmath_{!*} \colon \perv{Y} \to \perv{\overline{Y}}$ is exact. In fact it suffices for the above intersection cohomology group to vanish only for those local systems $\mathcal{L}$ which arise as the restriction of a local system on $S$.
\end{corollary}
\begin{proof}
The intermediate extension is exact if, and only if, for each $\mathcal{A} \in \perv{Y}$ the composition series of $\jmath_{!*}\mathcal{A}$ has no factors supported on $\overline{Y}-Y$. It is well-known that the intermediate extension has no non-zero subobjects (or quotients) supported on $\overline{Y}-Y$. However, Lemma \ref{local exact condition} implies that any factor supported on $\overline{Y}-Y$ would appear as a factor of a subobject supported on $\overline{Y}-Y$. Hence there are no non-zero such factors and the intermediate extension is exact.
\end{proof}

\begin{corollary}
\label{cs formula conditions}
Suppose that either
\begin{enumerate}
\item each stratum $S_k$ has finite fundamental group, or
\item $\ih{(\dim L -1)/2}{L;\mathcal{L}}=0$ for each link $L$ and each local system $\mathcal{L}$ on an open stratum of the link.
\end{enumerate}
Then (\ref{cs1}) is the canonical decomposition  of the class $[\beta]$. Moreover, under the second condition, (\ref{cs1}) can be more simply written as (\ref{cs2 order independent}).
\end{corollary}
\begin{proof}
By Lemma \ref{global exact condition} and Corollary (\ref{int ext exact}) either one of the conditions implies that each ${\jmath_k}_{!*}$ is exact. Then Corollary \ref{exact extension cor} implies that (\ref{cs1}) is the canonical decomposition, and hence independent of the choice of representative $\beta$.

Now suppose the second condition holds, so that the intermediate extension along the inclusion of {\em any} locally-closed union of strata is exact.  Choosing an anisotropic representative $\beta'$ Proposition \ref{anisotropic formula} implies that 
\[
[\beta] = [\beta'] =  \sum_{S \subset X} [ {\jmath_S}_{!*}\jmath_S^* \imath_S^{!*} \beta' ].
\]
Then $[ {\jmath_S}_{!*}\jmath_S^* \imath_S^{!*} \beta' ] = {\jmath_S}_{!*}\jmath_S^*[  \imath_S^{!*} \beta' ]$ because ${\jmath_S}_{!*}$ is exact, and
\[
[  \imath_S^{!*} \beta' ] = [\beta'] - \jmath_{!*}\jmath^*[\beta'] = [\beta] -  \jmath_{!*}\jmath^*[\beta] = [\imath_S^{!*}\beta]
\]
where $\jmath \colon X-\overline{S} \hookrightarrow X$ because $\jmath_{!*}$ is exact. The result follows.
\end{proof}
The conditions of this corollary are strong, but strong conditions are necessary. Example \ref{cs2 failure} in the next section shows that (\ref{cs2 order independent}) need not hold even when all strata are simply-connected.

\subsection{Perverse sheaves on rank stratifications}
\label{rankstrat}

In this section we use quiver descriptions of perverse sheaves on rank stratifications to illustrate some of our results. The main reference is \cite{braden1999}. 

We begin with the simplest non-trivial example, in which already one sees that one needs to take care in describing how Verdier duality translates to quiver descriptions of perverse sheaves. Let $X= \C$ stratified by the origin and its complement, and let $\jmath \colon \C-\{0\} \hookrightarrow  \C \hookleftarrow \{0\} \colon \imath$ be the inclusions. Then the category  $\perv{X}$ of perverse sheaves of vector spaces over a field $\mathbb{F}$ of characteristic zero is equivalent to the category of representations of the quiver
\begin{equation}
\label{C quiver}
\begin{tikzcd}
 0  \ar[bend left=15]{r}{c}    &   \ar[bend left=15]{l}{v}  1
\end{tikzcd}
\end{equation}
where $1+cv$ and $1+vc$ are invertible \cite[\S4]{MR804054}. Here the vertex $0$ corresponds to the perverse nearby cycles $\Psi_z$ and $1$ to the perverse vanishing cycles $\Phi_z$, where $z$ is a coordinate on $\C$. The arrow $c$ corresponds to the canonical map and $v$ to a variation map. In order to define the latter one needs to pick an orientation of { $\C^*$}, or equivalently a generator of the fundamental group of $\C^*$. The restriction of a perverse sheaf $\sh{E}$ to $\C^*$ is a shifted local system $\sh{L}[1]$, whose stalk $\sh{L}_1$ at the basepoint $1\in \C^*$ is the perverse nearby cycles, and whose monodromy with respect to the chosen generator is $\mu = 1+vc$. The restriction of the Verdier dual $D\sh{E}$ to $\C^*$ is the dual shifted local system $\sh{L}^\vee[1]$ whose monodromy with respect to the {\em reversed} generator is $\mu^*=1+c^*v^*$, where $*$ denotes the vector space dual. 

In order to give a simple description of duality, sufficient for the following examples, we restrict to the Serre subcategory of perverse sheaves with { unipotent} monodromy, \ie for which both $n_\Phi=cv$ and $n_\Psi=vc$ are nilpotent. This allows us to switch to an alternative description in terms of the logarithms $N_\Phi$ and $N_\Psi$ of the monodromies, by which we mean
\[
1+n =e^{N}
\]
in each case. We do so by replacing the variation arrow $v$ by $V = f(n_\Psi)v = vf(n_\Phi)$ where 
\begin{equation}
\label{log power series}
f(t) = \frac{\ln(1+t)}{t} = 1-\frac{t}{2}+\cdots.
\end{equation}
Verdier duality for { unipotent} perverse sheaves then corresponds to the duality 
\[
\begin{tikzcd}
 W_0  \ar[bend left=15]{r}{c}    &   \ar[bend left=15]{l}{V}  W_1
\end{tikzcd}
\quad
\longmapsto 
\quad
\begin{tikzcd}
 W_0^*  \ar[bend left=15]{r}{-V^*}    &   \ar[bend left=15]{l}{c^*}  W_1^*
\end{tikzcd}
\]
on quiver descriptions. Here we first switch to the reversed generator of the fundamental group by changing $V$ to $-V$, and then dualise. This fits with the duality of local systems, in which if $\sh{L}$ has { unipotent} monodromy $e^N$ then $\sh{L}^\vee$ has monodromy $e^{-N^*}$ with respect to the {\em same} generator of the fundamental group. If $\mathbb{F} \subset \C$ then under the Riemann--Hilbert correspondence this agrees (up to a Tate twist) with the description in terms of regular holonomic $D$-modules 
as given in \cite[Theorem 1.6,  Remark 1.7, and Theorem 2.2]{MR1045997}. Note that the usual biduality isomorphism $\chi: id\to D^2$ 
for quiver representations needs to be modified by a sign
at the vertex $0$, so that a symmetric bilinear form on a perverse sheaf corresponds to an (anti-)symmetric bilinear form at the vertex $1$
(respectively vertex $0$).

\begin{example}
\label{ie not exact 1}
Even in this simple example, intermediate extension does not induce a map of Witt groups. Let $\beta \colon \mathcal{B} \to D\mathcal{B}$ be the non-degenerate symmetric form
\[
\left(
\begin{array}{cc}
1 & 1 \\
0 & 1
\end{array}
\right)
\begin{tikzcd}
\F^2 
\ar[loop left]{l}
\ar{r}{\beta}
&
\F^2 \arrow[loop right]{r}
\end{tikzcd}
\left(
\begin{array}{cc}
1 & 0 \\
1 & 1
\end{array}
\right),
\qquad
\beta=\left(
\begin{array}{cc}
0 & 1 \\
-1 & 0
\end{array}
\right)
\]
on a two-dimensional shifted local system on $\C-\{0\}$. Clearly $\beta$ is metabolic with lagrangian the one-dimensional local system $\mathcal{A}$ with trivial monodromy, in particular $[\beta]=0$. The intermediate extensions of $\mathcal{A}$ and of $\mathcal{B}$ are
\[
\begin{tikzcd}
\F \ar[bend left=15]{r} &  0  \ar[bend left=15]{l} & \quad & \F^2 \ar[bend left=15, pos=0.45]{r}{(0\,1)} &  \F  \ar[bend left=15, pos=0.55]{l}{(1\,0)^*}
\end{tikzcd}
\]
respectively.  The subobject $\jmath_{!*}\mathcal{A}$ is isotropic for $\jmath_{!*}\beta$ but no longer lagrangian; the reduction $\ir{\jmath_{!*}\beta}{\jmath_{!*}\mathcal{A}}$ is the form $[1]$ on the point $0$. Applying (\ref{cs1}) to $\jmath_{!*}\beta$ we do not obtain the direct sum decomposition (\ref{witt group sum}) into forms on simple local systems. 

The intermediate extension $\jmath_{!*}$ is not exact, and moreover $[\jmath_{!*}\beta]\neq 0$ since any form Witt-equivalent to $\ir{\jmath_{!*}\beta}{\jmath_{!*}\mathcal{A}}$ must be a form on an object with an odd number of simple factors. Hence the intermediate extension does not induce a map of Witt groups. 
\end{example}

The above description can be generalised to perverse sheaves on a complex line bundle $L$ over a connected stratified space $X$. Stratify $L$ by the preimages of the strata of $X$ intersected with the zero section and its complement. Identify $X$ with the zero section, and let $\imath \colon X \hookrightarrow L$ be the inclusion, and $\jmath \colon L-X \hookrightarrow L$ its complement. A perverse sheaf with respect to this stratification is automatically \defn{monodromic} in the sense of \cite{MR804055} in that it is locally constant on the $\C^*$ fibres of the projection $L-X \to X$. The monodromy of such a perverse sheaf is an automorphism determined by a choice of orientation of $L$, equivalently of a generator of the fundamental group of $\C^*$. Perverse sheaves on $L$ are equivalent to representations of the quiver (\ref{C quiver}) but with values in the abelian category $\perv{L-X}$ rather than in vector spaces --- see \cite[Proposition 5.5]{MR804055}. When $L$ is a trivial line bundle this description corresponds to the one via perverse nearby and vanishing cycles for the projection $L \cong X\times \C \to \C$. The initial example considered above is the special one in which $X$ is a point.

A perverse sheaf $\sh{E}$ on $L-X$ splits as a direct sum $\sh{E}^u\oplus\sh{E}^{nu}$ of perverse sheaves with respectively unipotent and non-unipotent monodromy. For the non-unipotent summand
\[
{}^p\jmath_!\sh{E}^{nu} \cong {}^p\jmath_{!*}\sh{E}^{nu} \cong 
{}^p\jmath_*\sh{E}^{nu} 
\]
so that intermediate extension is exact on the full Serre subcategory of perverse sheaves with non-unipotent monodromy. For this reason we focus on the unipotent part. For this part Verdier duality can be described just as above, once a generator for the fundamental group of the $\C^*$ fibres is chosen, and the variation map is renormalised as before. 

Now we consider perverse sheaves on the rank stratification as in \cite{braden1999}. Let $V$ be an $n$-dimensional complex vector space. Stratify $\mathrm{End}(V)$ by the subspaces of endomorphisms of equal rank --- for $k=0,\ldots,n$ let
\[
S_k =\{x \in \mathrm{End}(V) \mid \mathrm{rank}(x) = n-k\}.
\]
Then $\codim S_k = k^2$ and each $S_k$ is connected with $\pi_1 S_0 \cong \Z$ and all other strata simply-connected. For $n>0$ \cite[Theorem 4.6]{braden1999} implies that the category $\perv{X_n}$ of perverse sheaves on the hypersurface 
\[
X_n=\{ x\in \mathrm{End}(V) \mid \mathrm{rank}(x) < n\} = S_1 \cup\cdots\cup S_n
\]
 of singular endomorphisms, constructible with respect to this stratification and  with coefficients in $\F$, is equivalent to the category of finite-dimensional $\F$-representations of the quiver with relations 
\[
\mathbb{A}_n = \left(
\left.
\begin{tikzcd}
 1  \ar[bend left=15, pos=0.6]{r}{c_1}    &   \ar[bend left=15, pos=0.4]{l}{v_1} \cdots  \ar[bend left=15]{r}{c_{n-1}}& n \ar[bend left=15]{l}{v_{n-1}} 
\end{tikzcd}
\right |
\begin{array}{l}
v_1c_1=0\\
c_kv_k = v_{k+1}c_{k+1} \ \text{for}\ k=1,\ldots,  n-2
\end{array}
\right).
\]
We write $\mu_k = 1+ v_kc_k$ for $k=1,\ldots, n-1$ and $\mu_n=1+c_{n-1}v_{n-1}$, and refer to these as the monodromies of the representation. By the conditions above each monodromy is { unipotent}, with $\mu_1=1$. For $n=1$ the quiver $\mathbb{A}_1$ corresponding to $\perv{\pt}$ has just one vertex and no arrows.

This equivalence between perverse sheaves and quiver representations is obtained in two steps. First one maps a perverse sheaf to its stratified Morse data, a vector space associated to each stratum $S_i$ together with the (micro-local) monodromy $\mu_i$, see \cite[Theorem 4.6 and Proposition 4.7]{braden1999}. This monodromy depends on a choice of generator of the micro-local fundamental group, which in each of these cases is infinite cyclic. The theory of micro-local perverse sheaves is then used in order to obtain the arrows in the quiver description, ultimately by reducing  to considering monodromic perverse sheaves on line bundles, see \cite[Proposition 2.8]{braden1999}. Since all micro-local monodromies are { unipotent}, we can renormalise the variation arrows as above, thereby modifying the identification of perverse sheaves with representations of the same quiver with relations $\mathbb{A}_n$.
Now the micro-local monodromies of the  perverse sheaves correspond to the following monodromies of the representation:
$\mu_k = e^{N_k}$ with $N_k= V_kc_k$ for $k=1,\ldots, n-1$ and $\mu_n=e^{N_n}$ with $N_n=c_{n-1}V_{n-1}$. 

For the remainder of this section we work with this modified identification. In this, Verdier duality on $\perv{X_n}$ corresponds to the functor mapping a representation
\[
\mathcal{A} = \left( 
\begin{tikzcd}
 A_1  \ar[bend left=15]{r}{c_1}   &   \ar[bend left=15]{l}{V_1} \cdots  \ar[bend left=15]{r}{c_{n-1}}& A_n \ar[bend left=15]{l}{V_{n-1}} 
\end{tikzcd}
\right)
\]
to the `dual' representation
\[
\mathcal{A}^*=\left(
\begin{tikzcd}
 A_1^*  \ar[bend left=15]{r}{-V_1^*}  &   \ar[bend left=15]{l}{c_1^*} \cdots  \ar[bend left=15]{r}{-V_{n-1}^*}& A_n^* \ar[bend left=15]{l}{c_{n-1}^*} 
\end{tikzcd}
\right).
\]
The usual biduality isomorphism $\chi: \id\to D^2$  for quiver representations needs to be modified by a sign $(-1)^{n^2-k^2}=(-1)^{n-k}$ at the vertex $k$, as for quivers with an involution \cite[\S3.2]{Young}, since the vector space associated to each stratum $S_k$  is given  as a normal Morse datum  shifted by the complex dimension $n^2-k^2$ of the stratum --- see  \cite[Corollary 5.1.4]{MR2031639}. Then a symmetric bilinear form on a perverse sheaf corresponds to a symmetric bilinear form at the vertex $k$ when $n-k$ is even and an anti-symmetric form when $n-k$ is odd. 

This differs from the description of duality in \cite[Proposition 4.8]{braden1999}, where the reversal of the generator of the fundamental  group is overlooked (as one can check in the $n=1$ case, which is the example considered at the beginning of this section). A further difference is that \cite{braden1999} state their results for perverse sheaves of complex vector spaces, however their methods apply more generally to perverse sheaves of vector spaces over any field $\mathbb{F}$. For the theory of micro-local perverse sheaves in this generality see \cite{MR2081221}. We must restrict to fields of characteristic zero in order to take logarithms of { unipotent} monodromies. 

We now explain how to understand intermediate extension and restriction in the quiver description of $\perv{X_n}$.
\begin{lemma}
\label{subcategories of perverse sheaves}
Under the above identification the diagram
\[
\begin{tikzcd}
\perv{S_k\cup \cdots \cup S_n} \ar{r}\ar{d}  & \perv{X_n} \ar{d}\\
\perv{S_k\cup \cdots \cup S_l} \ar{r}& \perv{S_1\cup \cdots \cup S_l}
\end{tikzcd}
\]
in which horizontal arrows are extensions by zero from closed unions of strata and vertical ones restriction to open unions, corresponds to the diagram
\[
\begin{tikzcd}
\langle \mathcal{A} \in \rep{ \mathbb{A}_n } \mid A_i=0 \ \text{for}\ i=1, \ldots k-1\rangle \ar{r}\ar{d}& \rep{\mathbb{A}_n}\ar{d} \\
\langle \mathcal{A} \in \rep{ \mathbb{A}_l } \mid A_i=0 \ \text{for}\ i=1, \ldots k-1\rangle \ar{r}& \rep{\mathbb{A}_l}.
\end{tikzcd}
\]
for $0<k\leq l \leq n$. Here the horizontal arrows are inclusions of full subcategories of quiver representations and the vertical ones arise from restricting a representation to the subquiver on vertices $1,\ldots, l$.
\end{lemma}
\begin{proof}
By \cite[Proposition 4.8]{braden1999} the restriction of a perverse sheaf to a normal slice to the stratum $S_l$ (and shifted by the complex dimension of $S_l$)
corresponds under the equivalence to the restriction of a representation of $\mathbb{A}_n$ to the subquiver on the vertices $1,\ldots, l$. 
In particular perverse sheaves on a normal slice can be identified with perverse sheaves on $X_l$. This remains the same under our modified identification. Perverse sheaves on the union $S_1\cup\cdots \cup S_l$ can also be identified with those on a normal slice to $S_l$ --- both are naturally equivalent to the category obtained by quotienting $\perv{X_n}$ by the Serre subcategory of perverse sheaves with vanishing Morse data on the strata $S_1, \ldots, S_l$. These correspond to representations $\mathcal{A}$ of the quiver $\mathbb{A}_n$ with $A_i=0$ for $i=1,\ldots, l$. The latter subcategory of perverse sheaves is the image of the extension by zero along the closed inclusion $S_{l+1} \cup \cdots \cup S_n \hookrightarrow X$. The result follows. 
\end{proof}

\begin{lemma}
\label{rank strat functors}
Let $\imath \colon S_k\cup\cdots\cup S_l \hookrightarrow S_1\cup\cdots\cup S_l$ and $\jmath \colon S_k\cup\cdots\cup S_l \hookrightarrow S_k\cup\cdots\cup S_n$ be the inclusions, where $0 < k\leq l\leq n$. Identify the representation
\[
\mathcal{A} = \left( 
\begin{tikzcd}
A_1 \ar[bend left=15]{r}{c_1} & A_2  \ar[bend left=15]{r}{c_2}   \ar[bend left=15]{l}{V_1} &   \ar[bend left=15]{l}{V_2} \cdots  \ar[bend left=15]{r}{c_l}& A_l \ar[bend left=15]{l}{V_l} 
\end{tikzcd}
\right)
\]
with a perverse sheaf on $S_1\cup\cdots\cup S_l$, and the representation
\[
\mathcal{B} = \left( 
\begin{tikzcd}
B_k \ar[bend left=15, pos=0.6]{r}{c_{k+1}} & B_{k+1}  \ar[bend left=15]{r}{c_{k+2}}   \ar[bend left=15, pos=0.4]{l}{V_{k+1}} &   \ar[bend left=15]{l}{V_{k+2}} \cdots  \ar[bend left=15]{r}{c_l}& B_l \ar[bend left=15]{l}{V_l} 
\end{tikzcd}
\right)
\]
with a perverse sheaf on $S_k\cup\cdots\cup S_l$. Then
\begin{enumerate}
\item $\pf{\imath^!} \mathcal{A} =
\left(
\begin{tikzcd}
\ker V_k \ar[bend left]{r} &  \ar[bend left]{l}  \cdots  \ar[bend left]{r}& \ker (V_k \cdots V_l ) \ar[bend left]{l}
\end{tikzcd}
\right)
$
\item $\pf{\imath^*} \mathcal{A} =
\left(
\begin{tikzcd}
\coker c_k \ar[bend left]{r} &   \ar[bend left]{l} \cdots  \ar[bend left]{r}& \coker (c_l \cdots c_k ) \ar[bend left]{l}
\end{tikzcd}
\right)
$
\item $\pf{\jmath_!} \mathcal{B} =
\left(
\begin{tikzcd}
B_k \ar[bend left=15]{r} & \cdots  \ar[bend left=15]{r}  \ar[bend left=15]{l} & B_l  \ar[bend left=15]{l} \ar[bend left=15, equal]{r}  & \cdots \ar[bend left=15]{l}{N_l} \ar[bend left=15, equal]{r} & B_l \ar[bend left=15]{l}{N_l}
\end{tikzcd}
\right)
$
\item $\pf{\jmath_*} \mathcal{B} =
\left(
\begin{tikzcd}
B_k \ar[bend left=15]{r} & \cdots  \ar[bend left=15]{r}  \ar[bend left=15]{l} & B_l  \ar[bend left=15]{l} \ar[bend left=15]{r}{N_l} & \cdots \ar[bend left=15, equal]{l} \ar[bend left=15]{r}{N_l} & B_l \ar[bend left=15, equal]{l}
\end{tikzcd}
\right)
$
\end{enumerate}
In each case the unlabelled upper and lower arrows are naturally induced respectively from the $c_i$ and the $V_i$. The natural morphisms $\pf{\imath^!} \mathcal{A} \to \mathcal{A} \to \pf{\imath^*} \mathcal{A}$ are given respectively by the evident inclusions and quotients, and the natural morphism $\pf{\jmath_!}\mathcal{B} \to \pf{\jmath_*}\mathcal{B}$ by the identity maps. The intermediate restriction therefore has
\[
\left( \imath^{!*}\mathcal{A} \right)_{k+i} = \im \left( \ker  (V_k\cdots V_{k+i} )\to A_{k+i} \to \coker (c_{k+i} \cdots c_k) \right)
\]
for $i=0, \ldots, l-k$, and the intermediate extension $\jmath_{!*} \mathcal{A}$ is the representation
\[
\begin{tikzcd}
A_k \ar[bend left]{r} & \cdots  \ar[bend left]{r}  \ar[bend left]{l} & A_l  \ar[bend left]{l} \ar[bend left, pos=0.6]{r}{N_l} & \im( N_l) \ar[bend left, pos=0.4]{r}{N_l} \ar[bend left, hookrightarrow]{l} & \cdots \ar[bend left, hookrightarrow]{l} \ar[bend left, pos=0.6]{r}{N_l} & \im(N_l)^{n-l} \ar[bend left, hookrightarrow]{l}.
\end{tikzcd}
\]
\end{lemma}
\begin{proof}
These results follow from the description of $\pf{\imath^!} \mathcal{A}$ and $\pf{\imath^*} \mathcal{A}$ as respectively the maximal subobject and quotient of $\mathcal{A}$ supported on $S_k\cup \cdots \cup S_l$,  and of $\pf{\jmath_!} \mathcal{B} $ and $\pf{\jmath_*} \mathcal{B} $ as respectively initial and terminal objects amongst all extensions of $\mathcal{B}$ to a perverse sheaf on $S_k\cup\cdots \cup S_n$.
\end{proof}

\begin{example}
\label{ie not exact 2}
Intermediate extension from a union of strata need not be an exact functor, even when all strata are simply connected. Consider the rank stratification for $n=3$, and specifically the inclusion $\jmath \colon S_1\cup S_2 \hookrightarrow S_1\cup S_2\cup S_3$. On the left below is a short exact sequence in $\perv{S_1\cup S_2}$, and on the right is the result of applying the intermediate extension $\jmath_{!*}$ to it:
\[
\begin{tikzcd}[row sep=large]
\F \ar[bend left=15]{r}{1} \ar{d}{1} &  \F\ar[bend left=15]{l}{0}  \ar{d}{(1\, 0)^*}
&&  
\F \ar[bend left=15]{r}{1}  \ar{d}{1}&  \F \ar[bend left=15]{l}{0}  \ar[bend left=15]{r} \ar{d}[pos=0.4]{(1\, 0)^*} & 0 \ar[bend left=15]{l} \ar{d}
\\
\F \ar[bend left=15, pos=0.55]{r}{(1\,0)^*} \ar{d} & \ar[bend left=15, pos=0.45]{l}{(0\,1)} \F^2 \ar{d}{ (0\, 1)}
&&
\F \ar[bend left=15, pos=0.55]{r}{(1\,0)^*} \ar{d} & \ar[bend left=15, pos=0.45]{l}{(0\,1)} \F^2 \ar[bend left=15, pos=0.45]{r}{(0\,1)} \ar{d}[pos=0.6]{ (0\, 1)}&  \F  \ar[bend left=15, pos=0.55]{l}{(1\,0)^*} \ar{d}
\\
0 \ar[bend left=15]{r} &  \F\ar[bend left=15]{l}
&&  
0  \ar[bend left=15]{r}& \F \ar[bend left=15]{r} \ar[bend left=15]{l}&  0\ar[bend left=15]{l}
\end{tikzcd}
\]
It is evident from the final column that the sequence on the right is no longer exact in the middle.
\end{example}
\begin{example}
\label{cs2 failure}
Let $\mathcal{B}$ be the perverse sheaf on $S_1\cup S_2\cup S_3$ in the middle row of the right hand diagram of  Example \ref{ie not exact 2} above, and let $\beta \colon \mathcal{B} \to D\mathcal{B}$ be the non-degenerate symmetric form 
\[
\begin{tikzcd}[row sep=large]
\F \ar[bend left=15, pos=0.55]{r}{(1\,0)^*} \ar{d}{1} & \ar[bend left=15, pos=0.45]{l}{(0\,1)} \F^2 \ar[bend left=15, pos=0.45]{r}{(0\,1)} \ar{d}{\alpha}&  \F  \ar[bend left=15, pos=0.55]{l}{(1\,0)^*} \ar{d}{-1}
\\
\F \ar[bend left=15, pos=0.55]{r}{(0\,-1)^*}  & \ar[bend left=15, pos=0.45]{l}{(1\,0)} \F^2 \ar[bend left=15, pos=0.45]{r}{(-1\,0)} &  \F \ar[bend left=15, pos=0.55]{l}{(0\,1)^*} 
\end{tikzcd}
\ \text{where}\
\alpha = \left( 
\begin{array}{cc}
0 & 1\\
-1& 0
\end{array}
\right).
\]
This has an evident isotropic subobject given by the simple object supported on the middle vertex. The corresponding isotropic reduction is the direct sum 
\[
I_{S_1\cup S_2 \cup S_3} \oplus -I_{S_3} = I_{X_3} \oplus -I_{\pt}
\]
 where we identify perverse sheaves and quiver descriptions, and write $I_Y$ for the intersection form on the intersection cohomology complex, with coefficients in $\F$, of a stratified space $Y$.

As before, let $S_3 \stackrel{\imath_2}{\longrightarrow} S_2\cup S_3\stackrel{\imath_1}{\longrightarrow} S_1\cup S_2\cup S_3$ be the closed inclusions, and let $\jmath_1 \colon S_1 \hookrightarrow S_1 \cup S_2\cup S_3$ and $\jmath_2 \colon S_2 \hookrightarrow S_2\cup S_3$ be the complementary open inclusions, and $\jmath_3 \colon S_3 \to S_3$ be the identity. Then
\[
\imath_1^{!*}\beta = 
\left( 
\begin{tikzcd}
0 \ar[bend left=15]{r} \ar{d}& \ar[bend left=15, pos=0.45]{l} \F \ar[bend left=15, pos=0.45]{l} \ar{d}{ -1}
\\
0 \ar[bend left=15]{r} & \ar[bend left=15, pos=0.45]{l} \F \ar[bend left=15, pos=0.45]{l}
\end{tikzcd}
\right)
=-I_{S_3}
\]
so that  $\imath_2^{!*}\imath_1^{!*}\beta=-I_{S_3}$ too. In contrast, $(\imath_1\circ \imath_2)^{!*} \beta=0$.  Hence (\ref{cs1}) is
\begin{align*}
[\beta] &=[{\jmath_1}_{!*}\jmath_1^* \beta] +[{\jmath_2}_{!*}\jmath_2^* \imath_1^{!*}\beta] + [{\jmath_3}_{!*}\jmath_3^*  \imath_2^{!*}\imath_1^{!*}\beta] \\
&=[I_{S_1\cup S_2\cup S_3}] +[0]+ [-I_{S_3}] \\
&= [I_{X_3}] - [I_{\pt}],
\end{align*}
in agreement with the isotropic reduction. However the formula (\ref{cs2 order independent}) is false in this case:
\[
\sum_{S \subset X} [ {\jmath_S}_{!*}\jmath_S^* \imath_S^{!*} \beta ] = [I_{X_3}]
\]
because $\imath_{S_3}^{!*}\beta = (\imath_1\circ \imath_2)^{!*} \beta=0$ and $\imath_{S_2}^{!*}\beta = (\imath_1\circ \jmath_2)^{!*} \beta= \jmath_2^* \imath_1^{!*}\beta = 0$ by the above calculations. Since $[I_{\pt}]\neq 0$ this does not agree with $[\beta] = [I_{X_3}] - [I_{\pt}]$. Therefore (\ref{cs2 order independent}) does not hold without further assumptions on the form, for instance that it is anisotropic. 
\end{example}

\subsection{Perverse sheaves on Schubert-stratified projective spaces}
\label{schubertstrat}

We consider a similar example but where the total space and the closures of each stratum are smooth. The main reference is \cite{MR1900761}, although we consider only the special case of projective spaces rather than all Grassmannians. The quiver description of perverse sheaves on Schubert-stratified projective spaces is well-known in the literature, e.g.\ \cite[Alternative proof of Proposition 2.9]{MR1862802} and \cite[Example 1.1]{stroppel}, however we need Braden's geometric approach in order to identify the action of Verdier duality.

Let $W$ be an $n$-dimensional complex vector space with a complete flag 
\[
W_0 \subset \cdots \subset W_n=W
\]
 of linear subspaces where $\dim_\C W_i = i$. Let $X=P(W)$ be the corresponding $(n-1)$-dimensional complex projective space with the Schubert stratification with strata $S_i = P(W_{n-i+1})- P(W_{n-i}) \cong \C^{n-i}$ for $i=1,\ldots,n$.

The category $\perv{X}$ of perverse sheaves with coefficients in the field $\F$ constructible with respect to the Schubert stratification is equivalent to the category of finite-dimensional $\F$-representations of the quiver with relations 
\[
\mathbb{A}'_n = \left(
\left.
\begin{tikzcd}
 1  \ar[bend left=15, pos=0.6]{r}{c_1}    &   \ar[bend left=15, pos=0.4]{l}{v_1} \cdots  \ar[bend left=15]{r}{c_{n-1}}& n \ar[bend left=15]{l}{v_{n-1}} 
\end{tikzcd}
\right |
\begin{array}{l}
v_1c_1=0\\
c_kv_k = v_{k+1}c_{k+1} \ \text{for}\ k=1,\ldots,  n-2\\
c_kc_{k-1} = 0 = v_{k-1}v_k \ \text{for}\ k=2, \ldots, n-1
\end{array}
\right).
\]
We write $\mu_k = 1+ v_kc_k$ for $k=1,\ldots, n-1$ and $\mu_n=1+c_{n-1}v_{n-1}$, and refer to these as the monodromies of the representation. By the conditions above each monodromy is { unipotent}, with $\mu_1=1$. This equivalence is a special case of \cite[Theorem 1.4.1]{MR1900761}. The general quiver description for Grassmannians given in \cite[\S1.3]{MR1900761} reduces to the description above in our case. 

The equivalence is obtained by a similar procedure as for the rank stratification case, except that one has to consider micro-local perverse sheaves through codimension $0$, $1$ and now also $2$. The argument for codimensions $0$ and $1$ is as before: first one maps a perverse sheaf to its stratified Morse data, a vector space at each stratum $S_i$ together with the (micro-local) monodromy $\mu_i$, see \cite[Proposition 4.3.1]{MR1900761}. This monodromy depends on a choice of generator of the micro-local fundamental group, which in each of these cases is again infinite cyclic. There is an arrow between vertices $i$ and $j$ if and only if the conormal spaces of $S_i$ and $S_j$ intersect in codimension $1$ \cite[Corollary 2.5.2]{MR1900761}, which for us is if and only if $|i-j|=1$. 

The relations $v_1c_1=0$ and $c_kv_k = v_{k+1}c_{k+1} $ for $k=1,\ldots,  n-2$ are deduced in the same way as for the rank stratification discussed in the previous section. The key technique is again to reduce to considering monodromic perverse sheaves on line bundles, see \cite[Lemma 3.4.1]{MR1900761}. The third type of relations $c_kc_{k-1} = 0 = v_{k-1}v_k$ for $k=2, \ldots, n-1$, c.f.\ \cite[(4) on p497]{MR1900761}, are obtained by considering codimension $2$ intersections of conormal spaces, i.e.\ for strata $S_i$ and $S_j$ with $|i-j|=2$, see \cite[Proposition 2.6.2]{MR1900761}. 

In order to describe Verdier duality in the quiver description we renormalise as in the rank stratification example. This is possible because as before all micro-local monodromies are { unipotent}. The quiver with relations  $\mathbb{A}'_n$ is unchanged, however the description of the (micro-local) monodromies is now 
$\mu_k = e^{N_k}$ with $N_k= V_kc_k$ for $k=1,\ldots, n-1$ and $\mu_n=e^{N_n}$ with $N_n=c_{n-1}V_{n-1}$. 
Note that as before the usual biduality isomorphism $\chi: \id\to D^2$ 
for quiver representations needs to be modified by a sign $(-1)^{n-k}$
at the vertex $k$ corresponding to a stratum $S_k$ of complex dimension $n-k$.
The descriptions of the six functors and of the intermediate extension and restriction remain the same because, as before, the description is compatible with restriction to a normal slice, see \cite[\S 4.2]{MR1900761}. We now work only with this modified identification. 

Example \ref{cs2 failure} transfers without change to this example in the $n=3$ case, i.e.\ for the Schubert stratification of $\C\P^2$. Here, not only are all three strata contractible but also their closures are smooth and simply-connected. Even under these strong conditions equation (\ref{cs2 order independent}) does not hold.

\begin{remark}
The path algebras of the  quivers with relations $\mathbb{A}_n$ and $\mathbb{A}'_n$, or their representation categories, appear in various other contexts:
\begin{enumerate}
\item as the Auslander algebra of $\C[x]/\langle x^n \rangle$ \cite{ploog};
\item in the braid group actions on categories studied in \cite{MR1862802};
\item  as convolution algebras related to hyperplane arrangements \cite[Example 4.6 and Theorem 4.8]{MR2680198};
\item as `hypertoric enveloping algebras' \cite[Example 4.11]{MR2964613}.
\end{enumerate}
As explained in these references, the representation categories of  $\mathbb{A}_n$ and $\mathbb{A}'_n$ are Koszul dual. The Koszul grading for $\mathbb{A}'_n$, and more generally for Braden's quiver description of perverse sheaves on Grassmannians, becomes visible only after a renormalisation similar to the one we use to understand Verdier duality, but this time using the square root of the power series (\ref{log power series}), see \cite[\S 5.7]{MR2521250}.
\end{remark}

\subsection{Relation to Cappell and Shaneson's work}
\label{relation to cs}
In the paper \cite{cs} Cappell and Shaneson introduce a notion of cobordism of self-dual complexes of sheaves { of vector spaces}, that is of objects $\mathcal{B} \in \constr{X}$ equipped with an isomorphism $\beta \colon \mathcal{B}\to D\mathcal{B}$, which is not assumed to have any symmetry properties. (Their definition of self-dual isomorphism involves a shift by $[\dim X]$, but we omit this because we are using the conventions of \cite{bbd} for indexing perverse sheaves rather than those of \cite{gm2}.) Let { $\Omega^{\pm}_{\textrm{\sc{CS}}}(X)$ denote the set of cobordism classes of constructible sheaf complexes with an (anti-)symmetric self-duality}.  The cobordism relation is generated by `elementary cobordisms' which arise from isotropic morphisms $\imath\colon  \mathcal{A} \to \mathcal{B}$. In the special case in which $\mathcal{A},\mathcal{B} \in \perv{X}$ and $\imath$ is a monomorphism, $\beta$ and $\ir{\beta}{\mathcal{A}}$ are elementarily cobordant. Thus there is a homomorphism
\[
{ \wgc{\pm}{\perv{X}} \to \Omega^{\pm}_{\textrm{\sc{CS}}}(X).}
\]
(Cappell and Shaneson do not discuss the structure of the set of cobordism classes, but \cite{MR1283567} shows that it is an abelian group under direct sum.)
{ Moreover, the homomorphism above is an isomorphism by \cite[Theorem 7.4]{youssin}.} This understood, their Theorem 2.1 states that the image of  (\ref{cs2 order independent}) holds in $\Omega_{\textrm{\sc{CS}}}(X)$. Example \ref{cs2 failure} above shows that this is incorrect ---  in that case there is a missing term corresponding to the class of the intersection form of a point --- and therefore that further conditions are required for their result. (Cappell and Shaneson work with compact spaces, so to be absolutely precise one should use the counterpart of Example \ref{cs2 failure} for Schubert-stratified projective spaces.) On \cite[p534]{cs}, in order to apply their (1.3), Cappell and Shaneson assume that ${}^p\imath_k^{!}\mathcal{A}=0$ implies ${}^p\imath_k^{*}\imath_{k+1}^!\mathcal{A}=0$. It is this which fails in Example \ref{cs2 failure}.

Cappell and Shaneson's decomposition is valid, and even lifts to the Witt group of perverse sheaves, when the form $\beta$ is anisotropic. It is also valid for any form $\beta$ on a sufficiently nice space  $X$, for instance when the second condition of  { Corollary \ref{cs formula conditions}} is satisfied. Another case in which it is valid is when the depth of $X$ is one, although in this case it may not correspond to the canonical decomposition. 

Let us suppose that we are in one of these `good' cases in which (\ref{cs2 order independent}) holds. Suppose further that $f\colon  Y \to X$ is a proper stratified map --- \ie a proper map such that the preimage of any stratum is a union of strata, and the restriction $f|_S \colon  S \to f(S)$ to any stratum is a locally trivial fibre bundle --- of Whitney stratified spaces with only even-dimensional strata. { Assume $Y$ has a dense top dimensional stratum which is oriented.} Then the intersection form $I_Y\colon \ic{Y} \to D \ic{Y}$ { of the corresponding intersection cohomology complex is non-degenerate in $\perv{Y}$ and is symmetric for $\dim Y \equiv 0\:(4)$ and anti-symmetric for $\dim Y \equiv 2\:(4)$}. Proper push-forward $f_*=f_!$ commutes with duality and so induces a map of Witt groups
{  $\wgc{\pm}{\perv{Y}} \to \wgc{\pm}{\perv{X}}$.} Hence (\ref{cs2 order independent}) yields
\[
[f_*I_Y] = \sum_S [{\jmath_S}_{!*}\jmath_S^*\imath_S^{!*}f_*I_Y].
\]
By proper base change $\imath_S^! f_* = f_* \ell_S^!$ and $\imath_S^* f_* = f_* \ell_S^*$, where $\ell_S \colon  f^{-1}\overline{S} \hookrightarrow X$. Hence 
\[
\imath_S^{!*}f_* = \im\!\left( H^0f_* \ell_S^!I_Y \to H^0f_* \ell_S^*I_Y \right)\!.
\]
 Section 4 of \cite{cs} uses this identification to interpret the local system $\jmath_S^*\imath_S^{!*}f_*I_Y$ on $S$ geometrically. The stalk is the middle-dimensional intersection cohomology of $f^{-1}N_x / f^{-1}L_x$ where $N_x$ is a normal slice to $S$ at $x\in X$, and $L_x = \partial N_x$ is the link. In this way one can obtain formul\ae\ for the Witt class, and thence the signature and $L$-class, of $Y$ as a sum of terms indexed by the strata of $X$, each with a natural geometric interpretation.

\subsection{Families of stratifications}

We make some brief remarks about Witt groups of perverse sheaves constructible with respect to a family of stratifications, rather than a fixed one. Let $\strat$ be a collection of stratifications of $X$ with only even dimensional strata, and such that any two stratifications admit a common refinement in $\strat$. Let $\pervc{\strat}{X}$ be the category of $\strat$-constructible perverse sheaves of $\F$-vector spaces on $X$. 
For example, $\strat$ might consist of a single stratification, or more interestingly $X$ might be a complex algebraic or analytic variety and $\strat$ the collection of all algebraic, respectively analytic, stratifications. 

We make $\strat$ into a poset with the ordering $S \leq \mathcal{S}' \iff \mathcal{S}'$ is a refinement of $\mathcal{S}$, \ie the strata of $\mathcal{S}$ are unions of strata of $\mathcal{S}'$. There is a fully-faithful inclusion
\[
\pervc{\mathcal{S}}{X} \hookrightarrow \pervc{\mathcal{S}'}{X}
\]
whenever $\mathcal{S} \leq \mathcal{S}'$. Moreover this inclusion commutes with duality and so induces a map of Witt groups.
\begin{proposition}
 Elements of the Witt group of $\strat$-constructible perverse sheaves $W(\pervc{\strat}{X})$ are represented by elements of $W( \pervc{\mathcal{S}}{X})$ for some stratification $\mathcal{S} \in \strat$; two such represent the same element if and only if they agree in the Witt group of perverse sheaves constructible with respect to { a common} refinement. In other words,
\[
W(\pervc{\strat}{X}) \cong \colim_{\mathcal{S} \in \strat} W( \pervc{\mathcal{S}}{X}).
\]
\end{proposition}
\begin{proof}
The universal property of the colimit induces a map 
\[
\colim_{\mathcal{S} \in \strat} W( \pervc{\mathcal{S}}{X}) \to W(\pervc{\strat}{X}).
\]
It is surjective since each class in $W(\pervc{\strat}{X})$ is represented by a form on a perverse sheaf which is constructible with respect to some particular stratification in $\strat$. If two such forms are equivalent, then the equivalence is realised by a finite sequence of isotropic reductions. So the forms are already equivalent in the Witt group of perverse sheaves constructible with respect to any sufficiently refined stratification for which all objects in this sequence are constructible. Hence the map is also injective.
\end{proof}

Say that $\strat$ is \defn{artinian} if the poset of closed unions of strata, considered as subspaces of $X$ ordered by inclusion, of all stratifications in $\strat$ is artinian. For example this holds if we work in the complex algebraic (respectively analytic) context with the collection of all algebraic (respectively analytic)  Whitney stratifications (on a compact analytic space). When this is the case the category $\pervc{\strat}{X}$ is both artinian and noetherian --- for algebraic stratifications this is \cite[Th\'eor\`eme 4.3.1]{bbd}, the general case is proved in a similar fashion. A simple object is an intermediate extension of an irreducible local system $\mathcal{L}$ on a stratum $S$. Two such, $\mathcal{L}$ on $S$, and $\mathcal{L}'$ on $S'$  are isomorphic if and only if there is a stratum $S''$, dense and open in both $S$ and $S'$, such that $\mathcal{L}|_{S''} \cong \mathcal{L}'|_{S''}$. Applying Corollary \ref{simples generate witt gp} we obtain
\begin{corollary}
If $\strat$ is artinian then each class in $W(\pervc{\strat}{X})$ has a decomposition into a sum of classes represented by forms on simple objects. The sum of terms represented by forms on a given isomorphism class of simple objects is well-defined. 
\end{corollary}
Irrespective of whether $\strat$ is artinian or not, one can apply Theorem \ref{sum thm} inductively to obtain formul\ae\  like (\ref{cs1}). If one decomposes in this way according to a stratification with respect to which a representative for the class is constructible then the summands will be represented by forms on intermediate extensions of local systems. In `good' cases (anisotropic forms or exact intermediate extensions) this sum will correspond to the canonical decomposition of the above corollary. 

\subsection{Unipotent nearby and vanishing cycles}

Let $X$ be a complex algebraic variety. Let $\perva{X}$ denote the algebraically constructible perverse sheaves on $X$. Fix an algebraic map $f\colon X \to \C$ and let $\imath \colon Y=f^{-1}(0) \hookrightarrow X$ and $\jmath \colon U=X-Y \hookrightarrow X$ be the inclusions. An important feature of this situation is that { the open inclusion  $\jmath\colon U \hookrightarrow X$ is an affine morphism}, which implies that $\pf{\jmath_!} = \jmath_!$ and $\pf{\jmath_*} = \jmath_*$.

There are exact functors
\[
\begin{tikzcd}
\perva{U} \ar{rr}{\ME} \ar{dr}[swap]{\NC} && \perva{X} \ar{dl}{\VC}\\
& \perva{Y}
\end{tikzcd}
\]
constructed in \cite{Beilinson1987}, see also the notes \cite{reich}. The functor $\ME$ is the \defn{maximal extension}, $\NC$ the \defn{unipotent nearby cycles}, and $\VC$ the \defn{unipotent vanishing cycles}. Here we follow the presentation of
\cite{Morel}, which is better adapted for the discussion of Verdier duality
\begin{remark}
In this section we work in the complex algebraic context without fixing a complex algebraic Whitney stratification.
However, all results apply to the case of a fixed Whitney stratification of $X$ in the complex algebraic or analytic context (with the same arguments), if $Y=f^{-1}(0)$ is a closed union of strata. For in that situation the corresponding constructible derived categories as well as the categories of perverse sheaves are stable under the functors  $\pf{\jmath_!} = \jmath_!$ and $\pf{\jmath_*} = \jmath_*$, as well as under $\NC$ and $\VC$ (see \cite[\S4.2.2 and \S6.0.4]{MR2031639}).
\end{remark}

Let $\Z(1)$ denote the orientation sheaf $\text{or}_{\C^*}$ of $\C^*$ and, by abuse of notation, also its stalk $\text{or}_{\C^*,1}\cong 2\pi i \Z$ at the chosen base point $1\in \C^*$. There is a natural representation $t$ of $\pi_1(\C^*,1)$ on  $\Z(1)$. A choice of  orientation of $\C^*$, equivalently of a generator $g\in \pi_1(\C^*,1)$, identifies $\Z(1)\cong \Z$ with the constant sheaf of integers with $t(g)=1$. As previously discussed, Verdier duality switches the chosen orientation to the opposite one with $t(g^{-1})=-t(g)$. In the following we therefore want to work without choosing an orientation.

For $n\geq 0$ let $\Z(n)=\Z^{\otimes n}$ and $\Z(-n)=\Z(-1) ^{\otimes n}$, where $\Z(-1)=\Z(1)^*$ is the dual local system.  Again we use the same notation for their stalks at the base point $1\in \C^*$, as well as for the corresponding local systems on $U=X-Y$ pulled back via $f: U\to \C^*$. Similarly  $-(n)=-\otimes_{\Z}\Z(n)$ denotes the corresponding Tate-twists of sheaves or stalks of $\mathbb{F}$-vector spaces, with $\mathbb{F}$ our base field of characteristic zero.

Consider for $p\geq 1$ the $p$-dimensional  $\mathbb{F}$-vector space  
\[
L^p=\mathbb{F}\oplus \mathbb{F}(-1)\oplus \dots \oplus  \mathbb{F}(1-p)
\]
together with the nilpotent morphism $N: L^p\to L^p(-1)$ given by  the matrix
\[
N=
\left(
\begin{array}{cccc}
0 & 1 & 0&   \\
0 & 0  & 1 & \\
0 & 0 & 0&  \\
 && &\ddots
\end{array}
\right).
\]
Let $\mathcal{L}^p$ be the corresponding local system on $\C^*$, with stalk $L^p$ in $1\in \C^*$, and monodromy action
\[
\mu(g)=e^{t(g)\cdot N}: L^p\to L^p
\]
 for $g\in  \pi_1(\C^*,1)$ any generator. For $p+q=n$ there is a short exact sequence
\begin{equation}
\label{local system ses}
0 \to \mathcal{L}^p \to \mathcal{L}^n \to \mathcal{L}^q(-p) \to 0
\end{equation}
where the maps are inclusion of the first $p$ coordinates, and projection onto the last $q$ coordinates. 

The unipotent nearby cycles of a perverse sheaf $\mathcal{A}$ on $U$ are defined by
\[
\imath_* \NC\mathcal{A}= \lim_{n\to \infty} \ker\left[ \jmath_! \left(\mathcal{A} \otimes f^*\mathcal{L}^n\right) \to \jmath_* \left(\mathcal{A} \otimes f^*\mathcal{L}^n\right)\right],
\]
where the map on the right hand side is the natural one. The kernel of this stabilises for sufficiently large $n$, and the limit denotes this stable kernel, see \cite[Corollary 3.2]{Morel}. The maximal extension of $\mathcal{A}$ is constructed similarly as
\[
\ME\mathcal{A} = \lim_{n\to \infty} \ker\left[ \jmath_! \left(\mathcal{A} \otimes f^*\mathcal{L}^n\right) \to\jmath_* \left(\mathcal{A} \otimes f^*\mathcal{L}^{n-1}\right) (-1)\right]
\]
where the map on the right is induced from the quotient in (\ref{local system ses}) with $q=n-1$, and once again the kernel stabilises for sufficiently large $n$, see \cite[Proposition 5.1]{Morel}. The action of $N: \mathcal{L}\to \mathcal{L}(-1) $ induces respective  actions  $\NC\mathcal{A}\to \NC\mathcal{A}(-1)$ and $\ME\mathcal{A}\to\ME\mathcal{A}(-1) $, which we also denote by $N$, and the same holds for the induced monodromy action $\mu(g)=e^{t(g)\cdot N}$ of a generator $g$ of $\pi_1(\C^*,1)$ on   $\NC\mathcal{A}$ and $\ME\mathcal{A}$.

Whereas the maximal extension functor commutes with Verdier duality \cite[Corollary 5.4]{Morel}, the unipotent nearby cycle functor commutes with Verdier duality only up to a Tate-twist \cite[Corollary 4.2]{Morel}:
\begin{equation}
D\left(  \NC\mathcal{A} \right) \cong  \NC \left( D(\mathcal{A})\right) (-1)\:.
\end{equation}
Moreover, there are two natural short exact sequences
\begin{equation}\label{dual SESs1} 
\begin{tikzcd}
0\ar{r} & \jmath_!\mathcal{A} \ar{r}{\alpha_- } & \ME \mathcal{A} \ar{r}{\beta_-} & \imath_* \NC\mathcal{A} (-1) \ar{r}&  0
\end{tikzcd}
\end{equation}
and
\begin{equation}\label{dual SESs2}  
\begin{tikzcd}
0\ar{r}& \imath_* \NC\mathcal{A}  \ar{r}{\beta_+} &  \ME \mathcal{A} \ar{r}{ \alpha_+}&  \jmath_*\mathcal{A}  \ar{r}&  0
\end{tikzcd}
\end{equation}
which are exchanged by duality \cite[Proposition 5.1, Corollary 5.4]{Morel}. The maps are induced from those in (\ref{local system ses}) for $(p,q)=(1,n-1)$ and $(n-1,1)$. The composite  $\alpha_+\circ \alpha_-$ is the natural map, and $\beta_-\circ\beta_+ = N$ \cite[Remark 5.6]{Morel}.
In particular the action $N: \NC\mathcal{A}\to \NC\mathcal{A}(-1)$ commutes with the duality isomorphism above:
$D\circ N\cong N\circ D$. This also holds without Tate-twists, if one chooses opposite generators of $\pi_1(\C^*,1)$ on both sides of this identification.
Otherwise a minus sign shows up, e.g.\ if one chooses on both sides the complex orientation of $\C^*$ as in \cite{MR1045997}.

The perverse unipotent vanishing cycles $\VC \mathcal{B}$ of $\mathcal{B}\in\perva{X}$ are defined to be (the restriction of) the cohomology 
$\imath^*H^0(-)$ of the complex
\[
\begin{tikzcd}
\jmath_!\jmath^*\mathcal{B} \ar{rr}{(\alpha_-,\gamma_-)^t }&&   \ME \jmath^*\mathcal{B} \oplus \mathcal{B} \ar{rr}{ (\alpha_+,-\gamma_+)} &&\jmath_*\jmath^*\mathcal{B} 
\end{tikzcd}
\]
 sitting in degrees $-1$ to $1$, where $\gamma_\pm$ are the unit and counit of the adjunctions. Note that the first, respectively last,  morphism in this complex is injective, respectively surjective, with its cohomology $H^0(-)$ supported on $Y$ (since its restriction to $X-Y$ is vanishing). That $\VC$ commutes with duality follows from the fact that duality interchanges the above two short exact sequences \cite[Remark 6.1]{Morel}. One also gets induced morphisms 
\[
\begin{tikzcd}
\NC\jmath^*\mathcal{B} \ar{r}{\text{can}}&  \VC \mathcal{B} \ar{r}{\text{Var}} & \NC\jmath^*\mathcal{B}(-1),
\end{tikzcd}
\]
of perverse sheaves on $Y$ with $N=\text{Var} \circ \text{can}$, so that $\text{can}$ and $\text{Var}$ are exchanged by duality \cite[Remark 6.1]{Morel}. Moreover, the category $\perva{X}$ can be described in terms of the gluing data
\cite[Theorem 8.1]{Morel}:
\[
\mathcal{B} \mapsto \left( \jmath^*\mathcal{B},  \VC \mathcal{B}, \text{can}, \text{Var} \right).
\] 
For example $\jmath_{!*}\mathcal{A}$ has the following gluing data description (see also \cite[Proposition 4.7]{reich}):
\begin{equation}\label{gluing-intermediate}
\left( \mathcal{A}, \im(N: \NC \mathcal{A} \to \NC \mathcal{A}(-1)), N, \text{incl}\right),
\end{equation}
with $N: \NC (\jmath^{*} \jmath_{!*}\mathcal{A})=\NC \mathcal{A} \to \NC \mathcal{A}(-1)=\NC (\jmath^{*} \jmath_{!*}\mathcal{A}) (-1) $ factorised as
\[
\begin{tikzcd}
\NC \mathcal{A} \ar{r}{N}& \im(N) \ar{r}{\text{incl}} & \NC \mathcal{A}(-1).
\end{tikzcd}
\]
Since the unipotent vanishing cycles and the maximal extension commute with duality they induce maps of Witt groups. We abuse notation by using the functors to denote these induced maps.
\begin{lemma}
\label{hypersurface splitting}
The map $[\beta] \mapsto \left( \jmath^*[\beta] , \VC [\beta] \right)$ is an isomorphism
\[
W(\perva{X}) \cong W(\perva{U}) \oplus W(\perva{Y})
\]  
with inverse $\left( [\beta], [\beta'] \right) \mapsto \ME [\beta] + \imath_* [\beta']$.
\end{lemma}
\begin{proof}
From the constructions $\VC \circ \imath_*$ and $\jmath^*\circ \ME$ are the identity. Therefore
\[
\ME [\beta] = \imath_*[\beta']
\]
implies $[\beta] = \jmath^* \ME[\beta] = 0$, and hence $[\beta']=\VC \imath_*[\beta']=0$ too. 

Given $[\beta] \in W(\perva{X})$ the form $\beta \oplus \ME \jmath^*\left(-\beta \right)$ is metabolic when restricted to $U$. Using $\jmath_!$ we can construct an isotropic subobject for this form from a lagrangian for the restriction. The reduction by this isotropic subobject will be supported on $Y$, so that 
\[
[\beta] - \ME \jmath^*[\beta] = \imath_* [\beta']
\]
for some $[\beta'] \in W(\perva{Y})$. We now show that $[\beta'] = \VC[\beta]$, or equivalently that $\VC \circ \ME =0$ on Witt groups. To see this recall that there is a functorial short exact sequence \cite[Corollary 7.2]{Morel}
\[
0\to \NC\to\VC\ME\to  \NC(-1) \to 0,
\]
so that the induced form  $\VC \ME[\beta]$ is metabolic. Therefore $\VC \ME[\beta]=0$ in the Witt group as claimed.
\end{proof}

We can relate the above decomposition to our earlier splitting results.
\begin{corollary}
\label{relating gluing and splitting}
For $[\beta: \mathcal{A}\to D(\mathcal{A})]\in W(\perva{U})$ the composite 
\[
\NC\beta \circ N \colon \NC \mathcal{A} \to \NC \mathcal{A}(-1)  \to D \NC \mathcal{A}
\] 
is symmetric and $[\jmath_{!*}\beta] = \ME [\beta] +\imath_* [\NC\beta \circ N]$. Similarly for $[\beta']\in W(\perva{X})$ we have
\[
[\imath^{!*}\beta'] = \VC [\beta'] - [\NC(\jmath^*\beta' )\circ N].
\]
\end{corollary}
\begin{proof}
It is easy to verify that 
\[
\NC\beta \circ N \colon \NC \mathcal{A} \to \NC \mathcal{A}(-1)  \to\NC (D(\mathcal{A})) (-1) \cong  D \NC \mathcal{A}
\]
is symmetric, since $N:  \NC \mathcal{A} \to \NC \mathcal{A}(-1)$ commutes with duality. Moreover,  from the description of intermediate extensions in terms of gluing data (\ref{gluing-intermediate}) one gets
\[
\VC [\jmath_{!*}\beta] = [\NC\beta \circ N].
\]
Hence by Lemma \ref{hypersurface splitting} $[\jmath_{!*}\beta] = \ME [\beta] + \imath_* [\NC\beta \circ N]$ so that
\[
[\imath^{!*}\beta'] = [\beta'] - [\jmath_{!*}\jmath^*\beta'] = \VC [\beta'] - [\NC(\jmath^*\beta') \circ N]
\]
as claimed.
\end{proof}
\begin{remark}
An alternative method of proof is to verify that the first equation is the splitting relation arising from the short exact sequences (\ref{dual SESs1}) and (\ref{dual SESs2}).
The second is the splitting relation for the following two exact sequences of perverse sheaves which are exchanged by duality \cite[Proposition 6.2]{Morel}:
\[
\begin{tikzcd}
\quad \quad \quad  \quad\quad
\NC (\jmath^*\mathcal{B}) \ar{r}{\text{can}}& \VC \mathcal{B} \ar{r} & H^0(\imath^*\mathcal{B}) \ar{r} & 0
\end{tikzcd}
\]
and
\[
\begin{tikzcd}
0 \ar{r}& H^0(\imath^!\mathcal{B}) \ar{r} & \VC\mathcal{B} \ar{r}{\text{Var}} & \NC (\jmath^* \mathcal{B}) (-1),
\end{tikzcd}
\]
with $H^0$ the corresponding perverse cohomology. 
\end{remark}

\end{document}